\documentclass{amsart}
\numberwithin{equation}{section}
\usepackage{amssymb}
\usepackage[all]{xy}
\let\cal\mathcal

\def\Uscr{{\cal U}}

\let\blb\mathbb
\def\CC{{\blb C}} 
\def\DD{{\blb D}}

\def \ZZ{{\blb Z}}

\def \HH{{\blb H}}

\def\id{\text{id}}
\def\Id{\operatorname{id}}

\def\Der{\operatorname{Der}}

\def\Lotimes{\overset{L}{\otimes}}

\def\Mod{\operatorname{Mod}}
\def\DGMod{\operatorname{DGMod}}
\def\mod{\operatorname{mod}}
\def\Gr{\operatorname{Gr}}

\def\rad{\operatorname {rad}}
\def\PC{\operatorname {PC}}

\def\Ext{\operatorname {Ext}}
\def\Hom{\operatorname {Hom}}

\def\End{\operatorname {End}}
\def\RHom{\operatorname {RHom}}

\def\Tr{\operatorname {Tr}}
\def\coker{\operatorname {coker}}
\def\ker{\operatorname {ker}}

\def\Tor{\operatorname {Tor}}
\def\End{\operatorname {End}}

\def\id{{\operatorname {id}}}

\def\r{\rightarrow}

\DeclareMathOperator{\Alg}{Alg}

\DeclareMathOperator{\res}{res}

\let\invlim\projlim

\newtheorem{lemma}{Lemma}[section]
\newtheorem{proposition}[lemma]{Proposition}
\newtheorem{theorem}[lemma]{Theorem}
\newtheorem{corollary}[lemma]{Corollary}

\newtheorem{lemmas}{Lemma}[subsection]
\newtheorem{propositions}[lemmas]{Proposition}
\newtheorem{theorems}[lemmas]{Theorem}
\newtheorem{corollarys}[lemmas]{Corollary}

\theoremstyle{definition}

\newtheorem{definition}[lemma]{Definition}

\newtheorem{definitions}[lemmas]{Definition}

{

}

\theoremstyle{remark}

\newtheorem{remark}[lemma]{Remark}
\newtheorem{remarks}[lemmas]{Remark}

\newdimen\uboxsep \uboxsep=1ex
\def\uboxn#1{\vtop to 0pt{\hrule height 0pt depth 0pt\vskip\uboxsep
\hbox to 0pt{\hss #1\hss}\vss}}

\def\uboxs#1{\vbox to 0pt{\vss\hbox to 0pt{\hss #1\hss}
\vskip\uboxsep\hrule height 0pt depth 0pt}}

\def\PC{\operatorname{PC}}
\def\PCGr{\operatorname{PCGr}}

\def\cont{\operatorname{cont}}
\def\Tw{\operatorname{Tw}}

\def\Alg{\operatorname{Alg}}
\def\Algc{\operatorname{Algc}}
\def\Cog{\operatorname{Cog}}
\def\Cogc{\operatorname{Cogc}}
\def\PCAlg{\operatorname{PCAlg}}
\def\PCAlgc{\operatorname{PCAlgc}}
\def\PCCog{\operatorname{PCCog}}

\def\DGComod{\operatorname{DGComod}}
\def\PCDGComod{\operatorname{PCDGComod}}
\def\DGMod{\operatorname{DGMod}}
\def\PCDGMod{\operatorname{PCDGMod}}

\def\cone{\operatorname{cone}}
\def\HH{\operatorname{HH}}
\def\HC{\operatorname{HC}}
\def\C{\operatorname{C}}
\def\CC{\operatorname{CC}}

\def\DR{\operatorname{DR}}
\def\DDer{\operatorname{{\blb D}er}}
\def\DR{\operatorname{DR}}
\def\ldb{\{\!\!\{}
\def\rdb{\}\!\!\}}
\title{Calabi-Yau algebras and superpotentials}
\author{Michel Van den Bergh}
\address{Universiteit Hasselt\\ Universitaire Campus\\ 3590 Diepenbeek}
\thanks{The author is a senior researcher at the FWO}
\email{michel.vandenbergh@uhasselt.be}
\begin{document}
\begin{abstract}
We prove that complete $d$-Calabi-Yau algebras in the sense of Ginzburg are derived 
from superpotentials.
\end{abstract}
\maketitle
\tableofcontents
\section{Introduction}
In this introduction we assume that $k$ is an algebraically closed field of
characteristic zero.  In the foundational paper \cite{Gi}
Ginzburg defines a $k$-algebra $A$ satisfying suitable finiteness
conditions to be \emph{$d$-Calabi-Yau} if there is a
quasi-isomorphism of complexes of $A$-bimodules
\[
\eta:\RHom_{A^e}(A,A\otimes A)\xrightarrow{\cong} {\Sigma^{-d}} A
\]
This property implies for example that the category of finite dimensional $A$-modules
is $d$-Calabi-Yau in the usual sense.  Sometimes one imposes the additional condition that~$\eta$ is
self dual but this appears to be automatic. See Appendix \ref{ref-C-112}. 

\medskip

In loc.\ cit.\ Ginzburg also introduces a particular class of
$3$-Calabi-Yau algebras which has found many applications in the theory of cluster algebras \cite{DWZ,DWZ2} and cluster
categories \cite{Amiot,Keller11,KY}.
Let $Q=(Q_0,Q_1)$ be a
finite quiver with vertices $Q_0$ and arrows $Q_1$. By definition a
\emph{superpotential} is an element $w$ of $kQ/[kQ,kQ]$. The \emph{Ginzburg
  algebra} $\Pi(Q,w)$ is the DG-algebra $(k\bar{Q},d)$ where $\bar{Q}$
is the graded quiver with vertices $Q_0$ and arrows
\begin{itemize}
\item the original arrows $a$ in $Q_1$ (degree 0);
\item opposite arrows $a^\ast$ for $a\in Q_1$ (degree -1);
\item loops $z_i$ at vertices $i\in Q
 _0$ (degree -2).
\end{itemize}
The differential is\footnote{The unusual sign in the definition of $da^\ast$ is an artifact of
our setup.}
\begin{align*}
da&=0&(a\in Q_1)\\
da^\ast&=-\frac{{}^\circ\partial w}{\partial a}& (a\in Q_1)\\
dz&=\sum_{a\in Q_1} [a,a^\ast] &
\end{align*}
for $z=\sum_i z_i$.
Here ${}^\circ \partial/\partial a$ is the so-called \emph{circular derivative}
\[
\frac{{}^\circ\partial w}{\partial a}=\sum_{w=uav}vu
\]
The homology in degree zero of a Ginzburg algebra is a so-called
 \emph{Jacobi algebra}.
Ginzburg shows that if the homology of $\Pi(Q,w)$ is concentrated in degree
zero then the associated Jacobi-algebra $H^0(\Pi(Q,w))$ is $3$-Calabi-Yau. In
\cite{Keller11}  a DG-version of this result is proved. 

\medskip

As observed in \cite[\S3.6]{Gi} the definition of $\Pi(Q,w)$ may be generalized to $d\ge 3$
(see also Lazaroiu's work in~\cite{Lazaroiu}). 
Let $Q$ be an arbitrary finite
graded quiver and let $\tilde{Q}$ be the corresponding \emph{double
  quiver} obtained from $Q$ by adjoining opposite arrows $a^\ast$ of degree
$-d+2-|a|$ for $a\in Q_1$ (there is a slight subtlety with loops which we gloss over,
see \S\ref{ref-10.3-39}). It is now well-known that
$N=k\tilde{Q}/[k\tilde{Q},k\tilde{Q}]$ is a Lie algebra
when equipped with the so-called \emph{necklace bracket}   $\{-,-\}$ \cite{LebBock,Ginzburg1,Kosymp}. Let $w\in N$ be such that $|w|=-d+3$ and
$\{w,w\}=0$ and let $\bar{Q}$ be obtained from $\tilde{Q}$ by adjoining loops 
$(z_i)_i$ of degree $-d+1$ as above.
Then the \emph{deformed DG-preprojective algebra\footnote{This is our own
      terminology.}} $\Pi(Q,d,w)$ is the DG-algebra
$(k\tilde{Q},d)$ with differential
\begin{align*}
da&=(-1)^{(|a|+1)|a^\ast|} \frac{{}^\circ\partial w}{\partial a^\ast}
& (a\in Q_1)\\
da^\ast&=-(-1)^{|a|}\frac{{}^\circ\partial w}{\partial a}& (a\in Q_1)\\
dz&=\sum_{a\in Q_1} [a,a^\ast]
\end{align*}
Again Ginzburg proves that if the homology of $\Pi(Q,d,w)$ is concentrated in degree
zero then $H^\ast(\Pi(Q,d,w))$ is $d$-Calabi-Yau. If $w$ depends only on the arrows in $kQ$ then the condition $\{w,w\}=0$ is vacuous and $\Pi(Q,d,w)$ is a so-called ``deformed Calabi-Yau completion'' of $kQ$ as introduced by Keller. See \cite[\S 6.2]{Keller11}.

\medskip

Results about Calabi-Yau algebras are often most conveniently proved under the hypothesis that the algebras are derived from superpotentials (see e.g.\
\cite{VdBdT1,VdBdT2}). It is therefore a natural question how
restrictive this hypothesis is.  Before dealing with this we note that
it is generally understood that the above definition of Calabi-Yau should
somehow be strengthened to include higher homotopy information in the
definition of~$\eta$.

A suitable strengthening of the Calabi-Yau property was suggested to
the author by Bernhard Keller. It is related to a dual property used
by Kontsevich and Soibelman in \cite{KS2}. 

We first observe that
$\eta$ can be interpreted as a class in the Hochschild homology group
$\HH_d(A)$. Then we make the following definition
\newtheorem*{definitionA}{Definition}
\begin{definitionA}
  A $d$-Calabi-Yau algebra $A$ is \emph{exact
    Calabi-Yau
\footnote{This is our own
      terminology.}
} if $\eta$ can be taken in the image of the Connes map
  $B:\HC_{d-1}(A)\r \HH_d(A)$.
\end{definitionA}

Even with this strengthening of the Calabi-Yau property, it is
probably only sensible to attempt a classification in sufficiently
local cases. In this paper we will discuss the \emph{complete}
case. That is, roughly speaking, we discuss topological algebras which
are quotients of quivers completed at path length. For technical
background see \S\ref{ref-4-0} and also \cite[Appendix]{KY}. Note that
the complete case encompasses the graded case which has been treated
in \cite{Bocklandt}. Indeed the category of graded algebras is
equivalent to the category of complete algebras equipped with a
$k^\ast$-action.

\medskip

The main results in this paper are the following.
\theoremstyle{plain}
\newtheorem*{theoremA}{Theorem A}
\newtheorem*{theoremB}{Theorem B}
\begin{theoremA} \def\thetheorem{A} (see Corollary \ref{ref-9.3-27}) A ``complete''
  $d$-Calabi-Yau algebra is exact $d$-Calabi-Yau.
\end{theoremA}
This result depends on a vanishing property for periodic cyclic homology which
follows from Goodwillie's classical result for nilpotent extensions. 
\begin{theoremB} \def\thetheorem{B} (see Theorem \ref{ref-10.2.2-37} and
  \S\ref{ref-10.3-39}) Assume $d\ge 3$. A 
``complete'' exact
  $d$-Calabi-Yau DG-algebra $A$ concentrated in degrees $\le 0$ such that
$A/\rad A$ is commutative, is
  quasi-isomorphic to a (completed) deformed DG-preprojective algebra
  $\Pi(Q,d,w)$ with the degrees of the arrows in $Q$ lying in the
  interval $[(-d+2)/2,0]$ and $w$ being a linear combination of paths of length at
least $3$.
\end{theoremB}
By combining Theorems A and B it follows in particular that a complete
$3$-Calabi-Yau algebra is a Jacobi algebra (indeed: in this case the degree conditions ensure that the arrows in $Q$ have degree zero and that $w$  only depends on $Q$). This result has been
announced several years ago by Rouquier and Chuang but so far the
proof has not been published. A proof in the graded case for algebras generated
in degree one has been given in~\cite{Bocklandt}. 
A proof has also been given by Ed Segal under a suitable strengthening of the
Calabi-Yau condition (see \cite[Theorem 3.3]{segal}).

Some of the ideas of this manuscript have been used by Davison in
\cite{Davison} where he shows that group algebras of compact
hyperbolic manifolds of dimension greater than one are \emph{not}
derived from superpotentials.

\medskip

We now give an outline of the content of this paper. Somewhat arbitrarily it is
divided into a main body and appendices.
Whereas in the body of the paper we often impose boundedness conditions on
our DG-algebras, and furthermore $k$ is often characteristic zero, we have avoided
making such restrictions in the appendices. 

\medskip

In \S\ref{ref-4-0}-\ref{ref-6-8} we
discuss pseudo-compact algebras, modules, bimodules,\dots and their
homological algebra.  Our approach is somewhat different from
\cite[Appendix]{KY} as it relies heavily on duality.
We also need the bar cobar formalism which to the best of the author's knowledge
has not been systematically developed in the pseudo compact setting (although this
turns out to be easy). 
In order to make the text not too heavy we have deferred most
details to Appendix \ref{ref-A-82}.

In \S\ref{ref-7-9} we briefly discuss cyclic homology and its extension to the
pseudo-compact case. We also remind the reader of the $X$-complex formalism due
to Cuntz and Quillen. 

In \S\ref{ref-8-17} we discuss the different notions of  Calabi-Yau algebras and
in \S\ref{ref-9-22} we prove Theorem A. 

In \S\ref{ref-10-28} we introduce deformed DG-preprojective algebras and in
\S\ref{ref-11-43} we prove Theorem B.

In \S\ref{ref-12-72} we prove that the Koszul dual of a pseudo-compact exact
Calabi-Yau algebra has a cyclic $A_\infty$-structure. 

In Appendix \ref{ref-B-106} we prove the technical result that the
pseudo-compact Hochschild complex really computes $A\Lotimes_{A^e} A$.

In Appendix \ref{ref-C-112} we prove that the morphism $\eta$ appearing in Ginzburg's
definition of a Calabi-Yau algebra is automatically self dual.

Finally in Appendix \ref{ref-D-113} we prove some results on the behaviour of
Hoch\-schild/cyclic homology under Koszul duality. 
\section{Acknowledgement}
The author wishes to thank Bernhard Keller for generously sharing his
insights on Calabi-Yau algebras and in particular for explaining his
strengthening of the Calabi-Yau property during a 2006 Paris visit.
In addition he thanks Bernhard Keller for technical help with the bar
cobar formalism.  

This paper was furthermore strongly influenced by
ideas of Ginzburg \cite{Gi}, Kontsevich and Soibelman \cite{KS2} and
Lazaroiu \cite{Lazaroiu}.

The author thanks Maxim Kontsevich for pointing out to him that ``exact'' Calabi-Yau
is a better terminology than ``strongly'' Calabi-Yau which was used
in the first version of this article.

Finally the author thanks the referee for his very thorough reading of the manuscript. 

\section{Notation and conventions}
Throughout $k$ will be  a ground field. Unadorned tensor products are over
$k$. We follow a version of the Sweedler convention in the sense that
an element $a$ of a tensor product is written as $a'\otimes a''$ (without
summation sign). 

Unless otherwise specified algebras have units and maps are unit preserving. Modules
are unital as well. 
A similar convention holds for coalgebras and comodules.

If we write $a\in A$ for $A$ graded then we implicitly assume that $a$
is homogeneous. The degree of $a$ is written as $|a|$.
The shift functor on complexes is written as~$\Sigma$. The elements of $\Sigma A$
are written as $sa$ for $a\in A$ with $|sa|=|a|-1$. $\Hom_A(-,-)$ denotes graded Hom's. If we only want degree zero maps we write $\Hom_{\Gr(A)}(-,-)$. 

To reduce clutter  all operations are implicitly completed when working with pseudo-compact objects. 
Indeed the completions are implicit in the monoidal structure on the underlying
category of pseudo-compact vector spaces.

Throughout $\Delta$ means comultiplication, $\mu$ means multiplication, $\epsilon$ 
is either the counit or coaugmentation
and $\eta$ is either the unit or augmentation. 

The $1$-form associated to a function $f$ is denoted by $Df$. This is to avoid confusion
with the differential on DG-objects which will be denoted by $d$. We regard
$D(-)$ as an operation of homological degree zero. In other words it commutes with~$d$. 

\section{Pseudo-compact objects}
\label{ref-4-0}
When dealing with completed path algebras the natural context to
work in is that of pseudo-compact vector spaces/algebras/modules etc\dots.
See \cite{Gabriel,KY,VdB19}. We briefly recall the salient
features of this setting. 

By definition a \emph{pseudo-compact $k$-vector space} is a linear
topological vector space which is complete and whose topology is
generated by subspaces of finite codimension.  We will denote the
corresponding category by $\PC(k)$.  The topology on finite
dimensional pseudo-compact vector spaces is necessarily discrete and
conversely a finite dimensional vector space with the discrete
topology is pseudo-compact.  In particular $k$ itself is naturally
pseudo-compact.

We have inverse dualities
\begin{equation}
\label{ref-4.1-1}
\begin{gathered}
\DD:\Mod(k)\r \PC(k)^\circ:V\mapsto \Hom(V,k)\\
\DD:\PC(k)\r \Mod(k)^\circ:W\mapsto \Hom_{\cont}(W,k)
\end{gathered}
\end{equation}
where we recall that for $V\in \Mod(k)$ the topology on $\DD V$ is generated
by the kernels of $\DD V\r \DD V'$ where $V'$ runs through the finite dimensional
subspaces of~$V$.

It follows that $\PC(k)$ is a coGrothendieck category. In particular
$\PC(k)$ has exact filtered inverse limits (axiom AB5${}^*$). Furthermore
since the dual of $\PC(k)$ is locally noetherian, the product of 
projectives in $\PC(k)$ is projective.

One checks that the functor forgetting the topology
\[
\PC(k)\r \Mod(k)
\]
is exact, faithful and commutes with inverse limits. However it does
not commute with infinite direct sums.

We will systematically use $\DD$ to transfer notions from $\Mod(k)$ to 
$\PC(k)$. Thus if $V,W\in \PC(k)$ then we put
\[
V\otimes W=\DD (\DD W\otimes \DD V)
\]
or concretely
\[
V\otimes W=\invlim_{V',W'} V/V'\otimes W/W'
\]
where $V'$, $W'$ run through the open subspaces of $V,W$ respectively. 
By construction~$\DD$ is compatible 
with the monoidal structure. 

Below we also need \emph{graded pseudo-compact
vector spaces}. These are simply $\ZZ$-graded objects in the category
$\PC(k)$, i.e.\ sequences of pseudo-compact objects $(V^i)_{i\in \ZZ}$
(sometimes written as a formal direct product $\prod_{i\in \ZZ} V^i$). We denote the corresponding category by $\PCGr(k)$.
Putting $\DD((V^i)_{i\in \ZZ})=(\DD V^{-i})_{i\in \ZZ}$ defines a duality
between $\PCGr(k)$ and $\Gr(k)$, the latter being the category of graded $k$-vector spaces.

 There is a natural
functor ``forgetting the grading'' which commutes with $\DD$
\[
(-)^u:\PCGr(k)\r \PC(k):(V_i)_i\mapsto \prod_i V_i
\]

The monoidal structure on $\PCGr(k)$
is given by
\begin{align*}
(V_i)_i\otimes (W_j)_j&=\DD (\DD ((W_j)_j)\otimes \DD ((V_i)_i))\\
&=\biggl(\,\,\prod_{i+j=k} V_i\otimes W_j\biggr)_k
\end{align*}
Using the monoidal structure on $\PC(k)$ it is possible to define
\emph{pseudo-compact $k$-algebras, modules, bimodules}, etc\dots which are
simply the corresponding objects in $\PC(k)$.
\begin{lemma}
\begin{enumerate}
\item The topology on a pseudo-compact $k$-algebra $A$ is generated by
  twosided ideals of finite codimension.
\item Let $M$ be a pseudo-compact
left module over $A$. Then the topology on $M$ is generated by submodules of finite
codimension.
\item Let $M$ be a pseudo-compact
bimodule over $A$. Then the topology on $M$ is generated by subbimodules of finite
codimension.
\end{enumerate}
\end{lemma}
\begin{proof} One may prove this directly using the continuity properties of
the multiplication. Alternatively one may use the following observation:
if $A$ is a pseudo-compact $A$-algebra then $\DD A$ is a coalgebra. (1) then
follows from the fact that a coalgebra is locally finite (every coalgebra is
a filtered direct limit of finite dimensional coalgebras) \cite{Mont2}. (2)(3) are proved
in the same way. 
\end{proof}
\begin{remark}
  A pseudo-compact $k$-algebra $A$ is traditionally defined as a
  linear topological $k$-algebra which is complete and whose topology
  is generated by left ideals of finite codimension
  \cite{Gabriel,KY,VdB19}. This is equivalent to the above
  definition.  Indeed such an $A$ is an object in
  $\PC(k)$ and as the multiplication on $A$ is continuous, it represents
an algebra object in $\PC(k)$.
Similar observations hold for modules and bimodules. 
\end{remark}
Let $A$ be a pseudo-compact $k$-algebra. The common annihilator of the simple
pseudo-compact $A$-modules is called the radical of $A$ and is denoted
by $\rad A$. We recall the following
\begin{proposition} \cite[Prop.\ IV.13]{Gabriel} The radical of $A$ coincides
with the ordinary Jacobson radical of $A$. 
\end{proposition}

To eliminate another source of confusion we include
the following lemma. 
\begin{lemma} \label{ref-4.4-2}
If $A$ is a graded pseudo-compact $k$-algebra then $A^u$
  is an ordinary pseudo-compact algebra and its topology is generated
  by ideals of the form $L^u$ where $L$ is a graded ideal 
in $A$ of finite total codimension. We also have
\[
(\rad A)^u=\rad (A^u)
\]
\end{lemma}
\begin{proof} This follows by applying $\DD $ to the standard dual statement for coalgebras. 
\end{proof}

If $A$ is a pseudo-compact $k$-algebra then the category of left
pseudo-compact $A$-modules is denoted by $\PC(A)$. The category of
pseudo-compact $A$-bimodules is usually denoted by $\PC(A^e)$ with
$A^e=A\otimes_k A^\circ$. We use similar concepts and notations in the
graded context.

A \emph{DG-algebra $A$ over $\PC(k)$ is said to be pseudo-compact} if it is
pseudo-compact as a graded algebra and the differential is
continuous. DG-(bi)modules are defined similarly. 
\begin{remark} The topology on a pseudo-compact DG-algebra $A$ has 
a basis given by twosided DG-ideals of finite codimension. To
see this consider the dual statement for coalgebras. This says
that any DG-coalgebra $(C,d)$ should be the union of finite dimensional DG-coalgebras.

Indeed if $C'$ is a finite dimensional graded sub-coalgebra in $C$
then $C'+dC'$ is a finite dimensional sub-DG-coalgebra.

Similar comments can be made about left modules and bimodules. 
\end{remark}
\section{Duality for modules/bimodules}
\label{ref-5-3}
Below we will fix a finite dimensional separable $k$-algebra $l$ which
will be used throughout as a ground ring.  As~$l$ is non-commutative
this creates some technical problems with duals.  As pointed out in
\cite{BSW} there are 4 sensible ways to define the dual of an
$l$-bimodule.  However as shown in loc.\ cit.\ these can all be
identified by fixing a \emph{trace} on $l$, that is, a $k$-linear map
$\Tr:l\r k$ such that the bilinear form $(a,b)\mapsto \Tr(ab)$ is
symmetric and non-degenerate.

For $V\in \Mod(k)$ put $V^\ast=\Hom_k(V,k)$.  By functoriality it is
clear that $(-)^\ast$ sends left $l$-modules to right $l$-modules and
vice versa.  The challenge is to make the 
$(-)^\ast$-operation compatible with the monoidal structure given by
the tensor product over $l$. It is explained in \cite{BSW}
how to do this.

For the rest of this paper we fix a trace $\Tr:l\r k$ 
and we let $\sigma=\sigma'\otimes
\sigma''\in l\otimes_k l$ be the corresponding Casimir element. Note that  $\Tr$ is unique up to multiplication with a central unit in $l$.
The following properties will be used frequently: for $a\in l$
we have 
\begin{align*}
\sigma'\otimes \sigma''&=\sigma''\otimes \sigma'\\
a\sigma'\otimes \sigma''&=\sigma'\otimes \sigma'' a
\end{align*}
We define morphisms of $l$-bimodules
\[
\pi:l^\ast\r l:\pi(\phi)=\sigma'\phi(\sigma'')
\]
\[
c_{W,U}:W^\ast\otimes_l U^\ast\r (U\otimes_l
W)^\ast:\phi\otimes\theta\mapsto (u\otimes w\mapsto
\theta(u\sigma')\phi(\sigma''w))
\]
for $U$ a right $l$-module and $W$ a left $l$-module. The latter
morphism is natural in $U$ and $W$. We claim that these maps are
compatible with tensor product as expressed in the following lemma.
\begin{lemma}
$\pi$ is a bimodule isomorphism  and furthermore
the following  diagrams are commutative
\[
\xymatrix{%
  W^\ast\otimes_l l^\ast \ar[r]^{c_{W,l}}:\ar[d]_{\Id_{W^\ast}\otimes \pi} &(l\otimes_l W)^\ast\ar@{=}[d]\\
  W^\ast\otimes_l l\ar@{=}[r] &W^\ast
 } 
\qquad \xymatrix{%
  l^\ast\otimes_l U^\ast\ar[r]^{c_{l,U}}\ar[d]_{\pi\otimes\Id_{U^\ast}} &( U\otimes_ll)^\ast\ar@{=}[d]\\
  l \otimes_l U^\ast\ar@{=}[r]& U^\ast
}
\]
In addition if  $V$ is an $l$-bimodule then the following diagram is commutative
\[
\xymatrix{ W^\ast\otimes_l V^\ast\otimes_l U^\ast \ar[d]_{\Id_{W^\ast}\otimes c_{V,U}}
\ar[rr]^{c_{W,V}\otimes\Id_{U^\ast}} 
&&
  (V\otimes_l W)^\ast\otimes_l U^\ast\ar[d]^{c_{V\otimes_l W,U}} 
\\
  W^\ast\otimes_l
  (U\otimes_l V)^\ast\ar[rr]_{c_{W,U\otimes_l V}}&& (U\otimes_l V\otimes_l W)^\ast }
\]
\end{lemma}
The maps $\pi$ and $c_{??}$ are compatible with the canonical linear maps $V\r V^{\ast\ast}$
in the following sense.
\begin{lemma}
The following diagrams are commutative. 
\[
\xymatrix{
l^\ast\ar[rr]^{\pi^\ast}\ar[dr]_\pi&&l^{\ast\ast}\\
&l\ar[ur]&
}
\qquad
\xymatrix{
U^{\ast\ast}\otimes_l W^{\ast\ast}\ar[rr]^{c_{U^\ast,W^\ast}} &&(W^\ast\otimes_l U^\ast)^\ast
\\
U\otimes_l W\ar[u]\ar[rr]&&(U\otimes_l W)^{\ast\ast}\ar[u]^{c_{W,U}^\ast}
}
\]
\end{lemma}
We will use these results mostly in the following sense.
\begin{proposition}
  The contravariant functors denoted by $\DD$ define inverse dualities as
  monoidal categories between $\PC(l^e)$ and $\Mod(l^e)$. These
  dualities are compatible with the appropriate actions on categories
  of left and right $l$-modules.
\end{proposition}
For an $l$-bimodule $U$ we define $U_l=U/[l,U]$ and we let $U^l$ be the $l$-centralizer
in $U$. It is easy to see that the map
\begin{equation}
\label{ref-5.1-4}
(-)^{\S}:U_l\r U^l:m\mapsto \sigma' m \sigma''
\end{equation}
is an isomorphism of $k$-vector spaces. We denote its inverse by $(-)^{\dagger}$, 

The operations $(-)_l$, $(-)^l$ are compatible with dualizing in the sense that we have
canonical identifications
\[
(U_l)^\ast=(U^\ast)^l\qquad (U^l)^\ast=(U^\ast)_l
\]
Let $U,V$ be $l$-bimodules. Then 
\[
U\otimes_{l^e} V\cong (U\otimes_l V)_l
\]
We have a canonical ``flip'' isomorphism
\begin{equation}
\label{ref-5.2-5}
\ss:(U\otimes_{l} V)_l\r (V\otimes_l U)_l:m\otimes n\mapsto n\otimes m
\end{equation}
We would like to understand how this plays with duality. We have a map
\[
(U^\ast\otimes_l V^\ast)_l\xrightarrow{c_{U,V}} ((V\otimes_l U)^{\ast})_{l}\cong ((V\otimes_l U)^l)^\ast
\]
Now it is not so clear from this how the flip acts on $((V\otimes_l U)^l)^\ast$. Another
way of saying this is that if we look at the pairing derived from $c_{U,V}$
\begin{equation}
\label{ref-5.3-6}
(U^\ast\otimes_l V^\ast)_l\times (V\otimes_l U)^l:(\phi\otimes \theta,v\otimes u)
\mapsto \phi(u\sigma')\theta(\sigma'' v)
\end{equation}
then we see that exchanging $\phi$ and $\theta$ is not compensated for by simply exchanging
$u$ and $v$.

However if we compose further
\[
(U^\ast\otimes_l V^\ast)_l\r ((V\otimes_l U)^l)^\ast\xrightarrow{(-)^{\S\ast}} ((V\otimes_l U)_l)^\ast
\]
then one verifies that the following diagram is commutative
\[
\xymatrix{
  (U^\ast\otimes_l V^\ast)_l\ar[r]\ar[d]_\ss &((V\otimes_l U)_l)^\ast\ar[d]^{\ss^\ast}\\
  (V^\ast\otimes_l U^\ast)_l\ar[r] &((U\otimes_l V)_l)^\ast
}
\]
Another way of saying this is that we have a well defined pairing
\begin{equation}
\label{ref-5.4-7}
 (U^\ast\otimes_l V^\ast)_l\times (V\otimes_l U)_l:(\phi\otimes \theta,v\otimes u)
\mapsto \phi(\sigma'_2 u\sigma'_1)\theta(\sigma''_1 v\sigma''_2)
\end{equation}
which is now symmetric under simultaneously exchanging $(\phi,\theta)$ and $(u,v)$. 
\section{Homological algebra in the pseudo-compact case}
\label{ref-6-8}
An \emph{augmented $l$-algebra}  is an algebra $A$ equipped with
$k$-algebra homomorphisms $l\xrightarrow{\eta} A\xrightarrow{\epsilon}
l$ such that $\epsilon\eta$ is the identity. We have a corresponding
decomposition $A=l\oplus \bar{A}$ where $\bar{A}=\ker \epsilon=\coker \eta$.
Note that an augmented $l$-algebra is entirely determined by
the algebra structure on $\bar{A}$ (which is $l$-linear but likely without unit).

Obviously
there are similar concepts in the graded and DG-setting as well as in
the coalgebra setting. 

Below we use the notation $\Alg(l)$ for the category of augmented
$l$-DG-algebras and $\Cog(l)$ for the category of augmented
$l$-DG-coalgebras.  An object $C$ in $\Cog(l)$ is said to be
\emph{cocomplete} if $l$ identifies with the coradical of $C$, i.e.\ the sum
of all simple coalgebras\footnote{The coradical is automatically
graded and equates the graded coradical.}
  in $C$. We denote the full
subcategory of $\Cog(l)$ consisting of cocomplete coalgebras by
$\Cogc(l)$. We use the notations $\PCAlg(l)$, $\PCAlgc(l)$, $\PCCog(l)$ for
the corresponding pseudo-compact notions, the definition of which is dual to
those of the non-topological notions (we can deduce their definition by applying the
functor~$\DD$). For example an $l$-DG-algebra $A\in \PCAlg(l)$
is in $\PCAlgc(l)$ iff $\bar{A}=\rad A$. 

For $C\in \Cog(l)$ let $\DGComod(C)$ be the  category of
$l$-DG-comodules over $C$. For $A\in \Alg(l)$ let $\DGMod(A)$ be the
category of DG-modules over $A$.  Write $\PCDGComod(C)$ and $\PCDGMod(A)$ for
the corresponding pseudo-compact notions. 

\medskip

All these categories are equipped with model structures and connected
with each other through Quillen equivalences \cite{hinich,
  Keller12,Lefevre,Positselski}. See Appendix \ref{ref-A-82}
for a survey and further details. Let $A\in \PCAlgc(l)$. The model structure on $\PCDGMod(A)$ is dual to \cite[\S8.2]{Positselski}.
One has
\begin{enumerate}
\item The \emph{weak equivalences} are the morphisms with an acyclic cone. 
\item The \emph{cofibrations} are the injective morphisms with cokernel which is
  projective when forgetting the differential.
\item The \emph{fibrations} are the surjective morphisms.
\end{enumerate}
An object is \emph{acyclic} if it is in the smallest subcategory of the homotopy category
of~$A$ which contains the total complexes of short exact sequences and is closed
under arbitrary products. There is also a characterization 
of weak equivalences in terms of the bar construction. See Appendix \ref{ref-A.4-97}.

A weak equivalence between objects in $\PCAlgc(l)$ is
strictly stronger than a quasi-isomorphism. A similar statement holds
for weak equivalences in $\PCDGMod(A)$.
However under suitable boundedness assumptions (algebras concentrated in degrees $\le 0$ and modules concentrated in degrees $\le N$) weak equivalence is the same
as quasi-isomorphism (see the dual statements to Proposition \ref{ref-A.1.2-88} and Lemma \ref{ref-A.2.1-92}).

\section{Cyclic homology}
\label{ref-7-9}
\subsection{Mixed complexes and Hochschild/cyclic homology}
\label{ref-7.1-10}
 We recall
that a \emph{mixed complex} is a graded vector space $U$ equipped with maps
$b,B:U\r U$ such that $|b|=1$, $|B|=-1$, $b^2=0$, $B^2=0$,
$bB+Bb=0$ \cite{Loday1}.

The \emph{Hochschild homology} $\HH_\ast(U)$, \emph{negative cyclic homology} $\HC^-_\ast(U)$, 
\emph{periodic cyclic homology} $\HC^{\text{per}}_\ast(U)$  and
\emph{cyclic homology} $\HC_\ast(U)$ of $U$
are defined as the homologies of the following complexes
\begin{align*}
\C(U)&=(U,b)\\
\CC^-(U)&=(U[[u]],b+uB)\\
\CC^{\text{per}}(U)&=(U((u)),b+uB)\\
\CC(U)&=(U((u))/U[[u]],b+uB)[-2]
\end{align*}
where $u$ is a formal variable of degree $+2$.  We see that
cyclic homology is computed as the homology of a \emph{sum total complex}
of a suitable double complex whereas negative cyclic homology
is obtained from a \emph{product total complex}. Periodic cyclic homology
is derived from a total complex which is a mixture between a sum and
a product total complex.

By definition a morphism between mixed complexes $(U,b,B)\r (V,b,B)$ is
 a \emph{quasi-isomorphism} if $(U,b)\r (V,b)$ is a quasi-isomorphism.  The
homotopy category of mixed complexes is obtained by inverting quasi-isomorphisms.
Since both products and sums are exact functors
it is clear that Hochschild homology and the different
variants of cyclic homology remain invariant under quasi-isomorphisms
of mixed complexes.

Let $A$ be a unital $l$-DG-algebra. Then the \emph{Hochschild mixed complex} of
$A$ is $(\C(A),b,B)$ where $(\C(A),b)$ is the (\emph{sum}) total complex of the
standard Hochschild double complex
\begin{equation}
\label{ref-7.1-11}
\cdots \xrightarrow{\partial} (A\otimes_l A\otimes_l A)_l \xrightarrow{\partial} (A\otimes_l A)_l\xrightarrow{\partial} A_l\r 0
\end{equation}
As usual $B$ denotes the Connes differential. The \emph{normalized Hochschild mixed
complex} $(\bar{\C}(A),b,B)$ is \cite[\S2.1.9]{Loday1} the sum total complex of
\[
\cdots \xrightarrow{\partial} (A\otimes_l \bar{A}\otimes_l \bar{A})_l \xrightarrow{\partial} (A\otimes_l \bar{A})_l\xrightarrow{\partial} \bar{A}_l\r 0
\]
for $\bar{A}=A/l$. The ordinary and normalized Hochschild mixed complex
are quasi-isomorphic (see \cite[Prop. 1.6.5]{Loday1}).

The Hochschild homology
of $A$ and the different variants of cyclic homology are obtained
from the mixed Hochschild complex and they are denoted by 
$\HH_\ast(A)$, $\HC^-_\ast(A)$, $\HC^{\text{per}}_\ast(A)$, $\HC_\ast(A)$
respectively. The corresponding complexes are denoted by $\C(A)$, $\CC^-(A)$,
$\CC^{\text{per}}(A)$ and $\CC(A)$.  Below we will also use the
\emph{reduced} versions of these complexes (and their homology) which are obtained by replacing
$\C(A)$ by $\C(A)/\C(l)$.

If $A\in \PCAlgc(l)$ then we define its
Hochschild/cyclic homologies through the pseudo-compact version of the
Hochschild mixed complex which amounts to taking
the \emph{product total complex} of \eqref{ref-7.1-11}. It is explained
in Appendix \ref{ref-B-106} why this is a sensible definition.
\begin{remarks} For a pseudo-compact $k$-algebra the pseudo-compact Hochschild homology
is very different from the ordinary Hochschild homology. The following
example was pointed out long ago to the author by Amnon Yekutieli. 

Let $A=\mathbb{C}[[t]]$. Then the Hochschild homology of $A$ is equal
to $\Tor_\ast^{A^e}(A,A)$. Put $Q=\mathbb{C}((t))$. Then we have
\[
\Tor_\ast^{A^e}(A,A)\otimes_A Q=\Tor_\ast^{Q^e}(Q,Q)
\]
Now $Q$ is simply a field of (very) infinite transcendence degree over $\mathbb{C}$.
By a suitable version of the HKR theorem we obtain
\[
\HH_\ast(A)\otimes_A Q=\bigoplus_i \bigwedge^i\Omega_{Q/\mathbb{C}}
\]
and thus $\HH_i(A)\neq 0$ for all $i$. 
On the other the Hochschild homology of $A$ as pseudo-compact $k$-algebra
is concentrated in degrees $0,1$ and is given by the continuous
version of the HKR theorem. 
\end{remarks}

\subsection{$X$-complexes}
\label{ref-7.2-12}
We now recall the \emph{$X$-complex formalism} as introduced in
\cite{CQ,Quillen2}.  By definition an
\emph{$X$-complex} is a quadruple $\Uscr=(U,V,\partial_0,\partial_1)$ where $U,V$
are DG vector spaces. $\partial_0:U\r V$, $\partial_1:V\r U$ are
graded maps (of degree zero) commuting with the differentials and
$\partial_0\partial_1=0$, $\partial_1\partial_0=0$. An $X$-complex $\mathcal{U}$
has
an \emph{associated mixed complex} $(M(\mathcal{U}),b,B)$ given by
\[
M(\mathcal{U})=\Sigma V\oplus U
\]
\begin{align*}
b&=\begin{pmatrix}
-d & 0\\
\partial_1 & d
\end{pmatrix}\\
B&=\begin{pmatrix}
0 & \partial_0\\
0 & 0
\end{pmatrix}
\end{align*}
The homological invariants of an $X$-complex are those of the associated
mixed complex. 

\medskip

One may associate an $X$-complex to 
a unital $l$-DG-algebra $A$.
Let\footnote{It is unfortunate that the symbol $\Omega$ is used both for differentials
and for the cobar construction....}
\begin{equation}
\label{ref-7.2-13}
\Omega^1 _lA=\ker (A\otimes_l A\xrightarrow{a\otimes b\mapsto ab} A)
\end{equation}
$\Omega^1 _lA$ is the target for the universal $l$-derivation $D:A\r
\Omega_l^1 A$ which sends $a$ to $Da\overset{\text{def}}{=}a\otimes
1-1\otimes a$ (here the assumption that $A$ is unital is used).

For an $A$-bimodule $M$ put $M_\natural=A\otimes_{A^e} M=M/[A,M]$. 
We define $\partial_0:A_l\r (\Omega^1_lA)_\natural$ 
by $\partial_0(\bar{a})=\overline{Da}$ and $\partial_1:(\Omega^1_lA)_\natural\r A_l$ 
by $\partial_1(\overline{aDb})=\overline{[a,b]}$.

Both $\partial_0$ and $\partial_1$ commute with the differentials
inherited from the DG-structure on $A$.  Furthermore it is easy to see
that $\partial_0\partial_1=0$, $\partial_1\partial_0=0$.  The
$X$-complex $X(A)$ of $A$ is defined as
$(A_l,(\Omega^1_lA)_\natural,\partial_0,\partial_1)$.

There is a canonical morphism of mixed
complexes
\begin{equation}
\label{ref-7.3-14}
\sigma:C(A)\r M(X(A))
\end{equation}
defined as follows 
\[
\xymatrix{ \cdots \ar[r] & (A^{\otimes_l 4})_l \ar[r]^\partial\ar[d] &
  (A^{\otimes_l 3})_l \ar[r]^\partial \ar[d]
& 
(A^{\otimes_l 2})_l \ar[r]^\partial\ar[d]^{
  \begin{smallmatrix}a\otimes b\\\downarrow \\a Db\end{smallmatrix} } & A_l \ar[r]\ar@{=}[d] & 0
\\
\cdots \ar[r] & 0 \ar[r] & 0\ar[r] &(\Omega^1_lA)_\natural\ar[r]_{\partial_1} & A_l \ar[r] &0
}
\]
The following is well known.
\begin{propositions} \label{ref-7.2.1-15} 
\cite{CQ} Assume that $A$ is a cofibrant  \cite[\S1.3]{Lefevre} $l$-DG-algebra. Then
  $\sigma$ is a quasi-isomorphism of mixed complexes.
\end{propositions}
\begin{proof} 
We may assume that  $A=T_l V$ where $V$ is equipped with an ascending filtration
$0=F_0V\subset F_1V\subset\cdots$ such that $d(F_{i+1}V)\subset T_l F_i V$.

We have a short projective resolution of $A$ as a bimodule
\begin{equation}
\label{ref-7.4-16}
0\r \Omega^1_lA\xrightarrow{\partial_1} A\otimes_l A\r A\r 0
\end{equation}
and from the fact that $A$ is cofibrant one deduces that $\Omega^1_lA$ is a cofibrant
$A$-bimodule. For further reference we note that here $\partial_1$  is given by
\[
\partial_1(aDb)=ab\otimes 1-a\otimes b
\]
Comparing this with the usual resolution given by the shifted bar complex we get
a morphism between these resolutions
\[
\xymatrix{
\cdots \ar[r] & A^{\otimes_l 4} \ar[r]\ar[d] & A^{\otimes_l 3} \ar[r] \ar[d]^{
  \begin{smallmatrix}a\otimes b\otimes c\\\downarrow \\a (Db)c\end{smallmatrix} }& 
A^{\otimes_l 2} \ar[r]\ar@{=}[d] & A \ar[r]\ar@{=}[d] & 0
\\
\cdots \ar[r] & 0 \ar[r] &\Omega^1_lA\ar[r]& A^{\otimes_l 2} \ar[r] & A \ar[r] &0
}
\]
Applying $A\otimes_{A^e}-$ we get a morphism of double complexes which
is a quasi-isomorphism on the row level (with the $A$-differential oriented
vertically). We cannot immediately conclude that the sum total complex
is acyclic. However putting a suitable filtration on $T_l V$ we may
reduce to the case that the differential on $A$ is zero and then it works.
\end{proof}
The $X$-complex formalism generalizes without difficulty to the case
that $A\in \PCAlgc(l)$. Since now the Hochschild complex is a product
total complex the analogue of Proposition \ref{ref-7.2.1-15} is valid
under the weaker assumption that $A$ is a tensor algebra when
forgetting the differential. This is explained by the fact that the
latter objects are precisely the cofibrant objects in $\PCAlgc(l)$
(see \S\ref{ref-A.4-97}).

\section{Calabi-Yau algebras}
\label{ref-8-17}
We recall some definitions for a $k$-DG-algebra $A$.
\begin{definition} $A$ is \emph{homologically smooth} if $A$ is a perfect $A$-bimodule.
\end{definition}
For $M\in D(A^e)$ put $M^D=\RHom_{A^e}(M,A^e)$. The functor $(-)^D$ defines an auto duality
on the category of perfect $A^e$-modules. 
\begin{definition} 
\label{ref-8.2-18}
$A$ is \emph{Calabi-Yau of dimension} $d$ 
if  $A$ is homologically smooth and there exists an isomorphism $\eta:A^D\r {\Sigma^{-d}} A$ in the derived category of $A^e$-modules.
\end{definition}
This is a version of Ginzburg's definition
of a Calabi-Yau algebra. See \cite[Definition 3.2.3]{Gi}.  Ginzburg
assumes in addition that $\eta^D=\eta$ but it turns out this is automatic
(see Appendix \ref{ref-C-112}).

Assume that $A$ is a DG-algebra and $M,N$ are
$A$-DG-bimodules with $M$ perfect.  Then in $D(k)$ we have
\begin{equation}
\label{ref-8.1-19}
\RHom_{A^e}(M^D,N)\cong M\Lotimes_{A^e} N
\end{equation}
Indeed to prove this we may assume $M=A\otimes A$ and then it is obvious (see also
\cite[Remark 8.2.4]{KS2}). We say that $\xi\in H_d(M\Lotimes_{A^e} N)$ is
\emph{non degenerate} if the corresponding map $\xi^{+}:M^D\r {\Sigma^{-d}} N$ is an isomorphism
(see Appendix \ref{ref-C-112} for more details).

If follows that the isomorphism $\eta$ in Definition \ref{ref-8.2-18} defines a non-degenerate element
of $\HH_d(A)=H_d(A\Lotimes_{A^e} A)$ and conversely any such non-degenerate element 
of $\HH_d(A)$ defines an isomorphism $A^D\r{\Sigma^{-d}} A$. 

Associated to $A$ there is the Connes long exact sequence for cyclic homology.
\begin{equation}
\label{ref-8.2-20}
\cdots \r \HH_{d+1}(A)\xrightarrow{I} \HC_{d+1}(A)\xrightarrow{S} \HC_{d-1}(A) \xrightarrow{B} \HH_d(A)\xrightarrow{I} \HC_d(A)\xrightarrow{S} \HC_{d-2}(A)\xrightarrow{B} \HH_{d-1}(A)\r \cdots
\end{equation}
The following strengthening of the notion of Calabi-Yau was suggested by Bernhard Keller.
It is similar to a dual notion defined by Kontsevich and Soibelman \cite[\S 10.2]{KS2}.
\begin{definition} 
\label{ref-8.3-21}
$A$ is \emph{exact Calabi-Yau of dimension $d$} if $A$ is
homologically smooth and $\HC_{d-1}(A)$ contains an element $\eta$ such that
$B\eta$ is non-degenerate. 
\end{definition}
These definitions and observations extend without difficulty to the case that
$A\in \PCAlgc(l)$. Following \cite{KY} we say that an object
in $D(A)$ is \emph{strictly perfect} if it is contained in the smallest thick subcategory
of $D(A)$ containing $A$. We say that~$A$ is (topologically) homologically smooth
if $A$ is strictly perfect in $D(A^e)$.

Thanks to the model structure on $\PCDGMod(A^e)$ introduced in
\S\ref{ref-6-8} we may define $\RHom_{A^e}(A,A^e)$ and so
Definition \ref{ref-8.2-18} is meaningful. Furthermore the derivation of
\eqref{ref-8.1-19} is still valid so the concept of a non-degenerate
element in $\HH_d(A)$ (temporarily defined as $H_d(A\Lotimes_{A^e}A)$)
makes sense.

Thanks to Proposition \ref{ref-B.1-107} $H_\ast(A\Lotimes_{A^e}A)$ may be computed using the pseudo-compact Hochschild
complex which is itself a mixed complex. Hence we have the corresponding long exact sequence
for cyclic homology and so Definition \ref{ref-8.3-21} makes sense as well. 

\section{Cyclic homology for pseudo-compact (non DG-)Calabi-Yau algebras}
\label{ref-9-22}
The following  is a pseudo-compact version of Goodwillie's theorem 
\cite[Thm III.5.1]{Goodwillie}.
\begin{theorem}
\label{ref-9.1-23}
Assume that $k$ has characteristic zero
and let $A\in \Algc(l)$ be concentrated in degree $\le 0$. Then
$\HC^{\text{per},\text{red}}_\ast(A)=0$.
\end{theorem}
\begin{proof}[Proof of Theorem \ref{ref-9.1-23}]
Let $I=\rad A$. We must prove that
\[
\HC^{\text{per}}(A)\r \HC^{\text{per}}(A/I)
\]
is an isomorphism.
As the $I$-adic filtration on $A$ is dual to the coradical filtration
on $\DD A$, all $A/I^n$ are pseudo-compact and furthermore $A=\invlim_n A/I^n$.

The functor $\invlim$ commutes with (completed) tensor product so we also have
\[
\C(A)=\invlim_n \C(A/I^n)
\]
Since $A$ is concentrated in degree $\le 0$ the same is true for $\C(A)$ and
hence the complex $(\C(A)((u)),b+uB)$ computing periodic cyclic homology  (see \S\ref{ref-7.1-10})
involves only products.  Since inverse limits commute with products we get
\[
(\C(A)((u)),b+uB)=\invlim_n (\C(A/I^n)((u)),b+uB)
\]
Now by Goodwillie's theorem \cite[Thm II.5.1]{Goodwillie} for nilpotent
extensions\footnote{The proof of this result extends without difficulty to pseudo-compact
DG-algebras concentrated in degrees $\le 0$}  the morphism of complexes $(\C(A/I^n)((u)),b+uB)\r (\C(A/I)((u)),b+uB)$ is a quasi-isomorphism.
As filtered inverse limits are exact in the category of pseudo-compact
vector spaces we deduce that 
$\invlim_n (\C(A/I^n)((u)),b+uB))\r (\C(A/I)((u)),b+uB)$ is a quasi-isomorphism as well. 
\end{proof}
\begin{corollary}
  \label{ref-9.2-24} Assume that $k$ has characteristic zero and let
  $A\in \PCAlgc(l)$ be concentrated in degree zero. Assume in
  addition
that $A$ is $d$-Calabi-Yau. Then
  $\HC^{\text{red}}_i(A)=0$ for $i\ge d$.
\end{corollary}
\begin{proof} The algebra $\Ext^*_A(l,l)$ is finite dimensional and
  symmetric with an invariant form of degree $d$ (see \cite[Lemma
  3.4]{Keller11}). In particular $\Ext_A^i(l,l)=0$ for $i>d$. 

Consider a minimal free resolution of $A$ as $A^e$-bimodule
\[
\cdots \r F_2\r F_1\r A^e\r A\r 0
\]
Tensoring on the right with $l$ and applying $\Hom_A(-,l)$ we see by Nakayama's
lemma that this resolution has length $d$. Hence $\HH_i(A)=0$ for $i>d$. 

By the reduced version of \eqref{ref-8.2-20} we deduce for $i\ge d$
\begin{equation}
\label{ref-9.1-25}
\HC_i^{\text{red}}(A)=\HC_{i+2}^{\text{red}}(A)=\cdots
\end{equation}
Now 
\begin{equation}
\label{ref-9.2-26}
\HC^{\text{per,red}}_i(A)=\invlim_n \HC_{i+2n}^{\text{red}}(A)
\end{equation}
The proof of this fact is even easier than in the non-pseudo-compact
case (see \cite[Prop.\ 5.1.9]{Loday1}) because inverse
limits are exact for pseudo-compact vector spaces.

Combining \eqref{ref-9.1-25}\eqref{ref-9.2-26} with Theorem \ref{ref-9.1-23} proves what we want. 
\end{proof}
\begin{corollary}
\label{ref-9.3-27}
Assume that $k$ has characteristic zero and let $A\in \PCAlgc(l)$ be
concentrated in degree zero. Then $A$ is $d$-Calabi-Yau if and only if
it is exact $d$-Calabi-Yau.
\end{corollary}
\begin{proof} Assume that $A$ is $d$-Calabi-Yau. It follows from
  Corollary \ref{ref-9.2-24} that $HC^{\text{red}}_d(A)=0$.  It now follows from the analogue
  of \eqref{ref-8.2-20} for reduced homology that any $\eta\in
  \HH_d(A)$ lifts to an element of $\HC_{d-1}^{\text{red}}(A)$ and hence
to an element of $\HC_{d-1}(A)$. This
  proves what we want.
\end{proof}

\section{Deformed DG-preprojective algebras}
\label{ref-10-28}
In this section we assume that $k$ has characteristic zero.
We will show that pseudo-compact exact
Calabi-Yau DG-algebras are obtained from superpotentials.  Combining  this
with Corollary \ref{ref-9.3-27} it then follows that pseudo-compact \hbox{(non DG-)}Ca\-labi-Yau 
algebras are derived from superpotentials (and vice versa).

We will give complete proofs but
our arguments
are certainly heavily inspired by
\cite{Gi,KS2,Lazaroiu}. More precise
references will be given below. 

\subsection{A reminder on non-commutative symplectic geometry}
We introduce some notions from
\cite{CBEG,Kosymp}. See also
\cite{LebBock,Ginzburg1,VdB33,VdB36}.  Let $A$ be a graded algebra over
$\PC(l^e)$. Put
\[
\DR_l(A)=T_A(\Omega_l^1 A)/[T_A(\Omega_l ^1A),T_A(\Omega_l ^1A)]
\]
Here $T_A(\Omega^1_l A)$ and $\DR_l(A)$ are considered as bicomplexes. The grading by
``form degree'' is denoted by $|\!|-|\!|$ and the grading derived from the $A$-grading
is denoted by $|-|$.

If $\omega$, $\omega'$
are differential forms then their commutator is defined as
\[
[\omega,\omega']=\omega\omega'-(-1)^{|\!|\omega'|\!||\!|\omega''|\!|+|\omega'||\omega''|}\omega'\omega
\]
The operator $D$ extends to a derivation $T_A(\Omega^1_l A)\r T_A(\Omega^1_l A)$
homogeneous for $|-|$ and of degree one for $|\!|-|\!|$.  It is easy to see
that $D$ descends to a linear operator on $\DR_l(A)$.

Put $\DDer_l (A)=\Der_{A/l}(A,A\otimes A)$. The elements of $\DDer_l (A)$ are referred to as \emph{double derivations}.

Let $\delta\in \DDer_l (A)$. Then the
 \emph{contraction} $i_\delta$ with $\delta$ defines a double derivation of degree $(-1,|\delta|)$ 
on $T_A(\Omega_l^1 A)$ with respect
the bigrading $(|\!|-|\!|,|-|)$.  More precisely
\begin{align*}
i_\delta&:T_A(\Omega_l^1 A)\r T_A(\Omega_l^1A) \otimes T_A(\Omega_l ^1A)
\end{align*}
is defined as follows: for $a\in A$ one has
\begin{align*}
i_\delta(a)&=0\\
 i_\delta(Da)&=\delta(a)
\end{align*}
Following \cite{CBEG} we put
\begin{equation}
\label{ref-10.1-29}
\iota_\delta(\omega)=(-1)^{|i_\delta(\omega)'||i_\delta(\omega)''|} i_\delta(\omega)'' i_\delta(\omega)'
\end{equation}
\begin{definitions} \label{ref-10.1.1-30} \cite{CBEG} An element
  $\omega\in \DR_l(A)$ with $|\!|\omega|\!|=2$  which is closed for $D$ is
  \emph{bisymplectic} if the map of $A$-bimodules
\[
\DDer_lA\r \Omega^1_lA:\delta\mapsto \iota_\delta \omega
\]
is an isomorphism. 
\end{definitions}
Assume that $\omega\in \DR_l(A)$ is bisymplectic.
Following \cite{CBEG} we define the \emph{Hamiltonian vector
field}  $H_a\in \DDer_l(A)$ corresponding to $a\in A$ via
\begin{equation}
\label{ref-10.2-31}
\iota_{H_a} \omega =Da
\end{equation}
and we put
\begin{equation}
\label{ref-10.3-32}
\ldb a,b\rdb_\omega=H_a(b)\in A\otimes A
\end{equation}
and
\[
\{a,b\}_\omega=\ldb a,b\rdb'_\omega\ldb a,b\rdb''_\omega
\]
For the degree of the operations involved one finds by \eqref{ref-10.2-31}
\[
|H_a|+|\omega|=|a|
\]
and hence by \eqref{ref-10.3-32}
\[
|\ldb a,b\rdb_\omega|=|a|+|b|-|\omega|
\]
In other words $\ldb-,-\rdb$ has degree $-|\omega|$. 

It is shown in \cite[App. A]{VdB33} that 
\[
\ldb -,-\rdb_\omega:A\otimes A\r A\otimes A
\]
is a so-called (graded) ``\emph{double Poisson bracket}''. We will not
precisely define this notion but we note that it implies that $\{-,-\}_\omega$
descends
to a Lie algebra structure on $A/[A,A]$ (this is the Kontsevich bracket,
see \cite{Kosymp}) and furthermore that it 
defines an action of $A/[A,A]$ on $A$ by derivations. For use below let
us give the precise sign involved in exchanging the arguments of a 
double Poisson bracket (see \cite[\S2.7]{VdB33}).
\begin{equation}
\label{ref-10.4-33}
\ldb a,b\rdb_\omega=-(-1)^{(|a|-|\omega|)(|b|-|\omega|)}(-1)^{|\ldb a,b\rdb'_\omega||\ldb a,b\rdb''_\omega|}\ldb b,a\rdb_\omega''\otimes \ldb b,a\rdb_\omega'
\end{equation}
Now assume that  $V\in
\PC(l^e)$ is finite dimensional and let $\eta\in (V\otimes_l V)_l$ be a
non-degenerate element of degree $t$ in the sense that $\eta^+:V^D\r \Sigma^t V$
is an isomorphism (see Appendix \ref{ref-C-112}), which is furthermore
anti-symmetric for $\ss$ (see \S\ref{ref-5-3}).  Then it is easy
to see that
\[
\omega_\eta\overset{\text{def}}{=}\frac{1}{2} D\eta'D\eta''
\]
defines a bisymplectic form of degree $t$ on $A=T_lV$. We note the following. 
\begin{lemmas}
\label{ref-10.1.2-34}
For all $a\in T_lV$ one has
\[
\{a,\sigma'\eta\sigma''\}_{\omega_\eta}=0
\]
where $\sigma'\eta\sigma''\in (V\otimes_l V)^l$ (see \eqref{ref-5.1-4}) is viewed
as an element of $T_l V$. 
\end{lemmas}
\begin{proof}
\begin{align*}
\ldb a,\sigma'\eta\sigma''\rdb_{\omega_\eta}&=\ldb a,\sigma'\eta\sigma''\rdb_{\omega_\eta}'\otimes \ldb a,\sigma'\eta\sigma''\rdb_{\omega_\eta}''\\
&=(-1)^{(|a|+t)(|\eta|+t)} (-1)^{|\ldb a,\sigma'\eta\sigma''\rdb_{\omega_\eta}'||\ldb a,\sigma'\eta\sigma''\rdb_{\omega_\eta}''|}
\ldb \sigma'\eta\sigma'',a\rdb_{\omega_\eta}''\otimes \ldb \sigma'\eta\sigma'',a\rdb_{\omega_\eta}'
\end{align*}
(see \eqref{ref-10.4-33} for the signs which are not important here). Thus
\begin{equation}
\label{ref-10.5-35}
\{a,\sigma'\eta\sigma''\}_{\omega_\eta}=\pm \iota_{H_{\sigma'\eta\sigma''}} Da
\end{equation}
To continue we make the following claim
\[
H_{\sigma'\eta\sigma''}=-\Delta
\]
where $\Delta$ is the canonical double $l$-derivation which sends $f$
to $[f,\sigma'\otimes\sigma'']=f\sigma'\otimes \sigma''-\sigma'\otimes
\sigma ''f$.

To prove this claim we use the fact that the defining equation for $H_{\sigma'\eta\sigma''}$ is
\[
\iota_{H_{\sigma'\eta\sigma''}}\omega_\eta=D(\sigma'\eta \sigma'')
\]
We compute
\begin{align*}
2\iota_{\Delta}\omega&=(-1)^{|\eta'||\eta''|}\sigma'' (D\eta'')\eta' \sigma' 
-\sigma'' \eta' (D\eta'')\sigma'
-\sigma'' (D\eta') \eta''\sigma'
+(-1)^{|\eta'||\eta''|}\sigma''\eta'' (D\eta') \sigma'\\
&=-\sigma'' (D\eta')\eta'' \sigma' 
-\sigma'' \eta' (D\eta'')\sigma'
-\sigma'' (D\eta') \eta''\sigma'
-\sigma''\eta' (D\eta'') \sigma'\\
&=-2\sigma'' D(\eta'\eta'')\sigma'\\
&=-2D(\sigma'\eta\sigma'')
\end{align*}
where in the second line we have used anti-symmetry of $\eta$ and in the last
line the symmetry of $\sigma$. We get indeed
$H_{\sigma'\eta\sigma''}=-\Delta$. Plugging this into \eqref{ref-10.5-35} we find
\begin{align*}
\{a,\sigma'\eta\sigma''\}_{\omega_\eta}&=\pm \iota_{\Delta} Da\\
&=\pm(\sigma'' a\sigma'-\sigma''a\sigma')\\
&=0\qed
\end{align*}
\def\qed{}\end{proof}

\subsection{Deformed DG-preprojective algebras}
We now discuss a class of DG-algebras introduced by Ginzburg
in \cite[\S3.6]{Gi} and (somewhat implicitly) by
Lazaroiu in \cite[\S5.3]{Lazaroiu}. We introduce things in the pseudo-compact
context but of course the definition works just as well in the non-topological case. 

Suppose we have the data $(V_c,\eta,w)$ where 
\begin{enumerate}
\item $V_c\in \PC(l^e)$ is finite dimensional.
\item $\eta\in (V_c\otimes_l V_c)_l$ is non-degenerate and anti-symmetric under
$\ss$.
\item $w\in T_lV_c/[T_lV_c,T_lV_c]$ satisfies $|w|=|\eta|+1$ and $\{w,w\}_{\omega_\eta}=0$.
\end{enumerate}
The \emph{deformed DG-preprojective algebra} $\Pi(V_c,\eta,w)$ associated to
this data is the augmented pseudo-compact $l$-DG-algebra $T_l(V_c+zl)$ where $z$
is $l$-central and the differential is given by 
\begin{align*}
dz&=\sigma'\eta\sigma''\\
df&=\{w,f\}_{\omega_\eta}\qquad \text{for $f\in T_l V_c$}
\end{align*}
It is easy to see that this defines an honest DG-algebra. Indeed to verify
$d^2=0$ we have to check $d^2z=0$ and $d^2f=0$ for $f\in V_c$. We find
\[
d^2z=d(\sigma'\eta\sigma'')=\{w,\sigma'\eta\sigma''\}_{\omega_\eta}=0
\]
using Lemma \ref{ref-10.1.2-34} and
\begin{equation}
\label{ref-10.6-36}
d^2f=\{w,\{w,f\}_{\omega_\eta}\}_{\omega_\eta}=\frac{1}{2}\{\{w,w\}_{\omega_\eta},f\}_{\omega_\eta}=0
\end{equation}
\begin{remarks}
The element $w$ in the definition of a deformed DG-preprojective algebra is commonly called 
a \emph{superpotential}. 
\end{remarks}
The reason for introducing deformed DG-preprojective algebras is the following.
\begin{theorems}
\label{ref-10.2.2-37}
Assume that $A\in \Algc(l)$ is concentrated in degrees $\le 0$ and $d\ge 3$. Then the
following are equivalent.
\begin{enumerate}
\item $A$ is exact $d$-Calabi-Yau.
\item $A$ is weakly equivalent to an algebra of the form $\Pi(V_c,\eta,w)$
with $V_c$ concentrated in degrees $[-d+2,0]$, $|\eta|=-d+2$, $|w|=-d+3$ and
$w$ contains only cubic terms and higher.
\end{enumerate}
\end{theorems}
A related result has been proved by Ginzburg in the non-DG case
\cite[Thm 3.6.4]{Gi}. We will prove Theorem \ref{ref-10.2.2-37} in
\S\ref{ref-11-43}.
\begin{remarks} 
\label{ref-10.2.3-38}
The actual choice of the trace $\Tr:l\r k$ is
  irrelevant for the definition of $\Pi(V_c,\eta,\omega)$. Indeed
  different traces lead to $\sigma$'s differing by a
  central element  of $l$ which can be absorbed into
  $z$.
\end{remarks}
\begin{remarks} The hypothesis in Theorem \ref{ref-10.2.2-37} that $A$ is concentrated in
degrees $\le 0$ may be inconvenient in some cases. We may relax it as follows.
We say that an $l$-bimodule $W$ has \emph{no oriented cycles of strictly 
positive degree} if
 $((T_l W)_l)_{>0}=0$.  One may prove a variant of Theorem \ref{ref-10.2.2-37} under
the hypothesis that $A$ (as $l$-bimodule) has no oriented cycles of strictly
positive degree. The condition that $V_c$ lives in degrees $[-d+2,0]$ should
then also be replaced by the condition that $V_c$ has no oriented cycles of
strictly positive degree. 
\end{remarks}
\subsection{The quiver case}
\label{ref-10.3-39}
In this section we discuss quivers. \emph{We implicitly assume that
all path algebras  are completed at path length.}

We assume that $k$ is algebraically closed (still of
characteristic zero) and $l=\sum_{i=1}^n ke_i$ for central orthogonal
idempotents $(e_i)_i$. 
If we are given $(V_c,\eta)$ where $V_c$ is a
finite dimensional graded $l$-bimodule and $\eta$ is a non-degenerate
anti-symmetric element in $(V_c\otimes_l V_c)_l$ then we may write
\[
\eta=\sum_{ijt} [x_{ij}^{t},x_{ji}^{t\ast}]\qquad \text{mod $[l,-]$}
\]
where $\{x_{ij}^{t},x_{ji}^{t\ast}\}$ represents a homogeneous
basis for $e_i V_c e_j$ with $|x_{ij}^{t}|\ge |\eta|/2$. We assume that the
$x_{ij}^{t},x_{ji}^{t\ast}$ are all distinct except for the case $i=j$
where we assume $x_{ii}^{t}=x_{ii}^{t\ast}$ if
$|x_{ii}^{t}|=|x_{ii}^{t\ast}|=|\eta|/2$ is odd.

As usual the $x_{ij}^{t}$, $x_{ji}^{t\ast}$ may be regarded as arrows
in a quiver $\tilde{Q}$ with vertices $\{1,\ldots,n\}$ and it easy to
check that in that case $\{-,-\}_{\omega_\eta}$ is precisely the
\emph{necklace bracket} on $T_l V_c=k\tilde{Q}$
\cite{LebBock,Ginzburg1,Kosymp}.
For a pictorial presentation of the necklace bracket as an operation
\[
k\tilde{Q}/[k\tilde{Q},k\tilde{Q}]\times k\tilde{Q}\r k\tilde{Q}
\]
see \cite[Prop.\ 6.4.1]{VdB33}.

If we let $Q$ be the graded quiver with arrows $x_{ij}^{t}$ of degree $\ge |\eta|/2$ then $\tilde{Q}$ is
obtained from $Q$ by the following procedure.
\begin{enumerate}
\item For every non-loop $a\in Q$ we adjoin an arrow $a^\ast$ in the
opposite direction such that $|a^\ast|+|a|=|\eta|$. 
\item We do the same for loops, \emph{except} for a loop $a$ of odd degree
such that $|a|=|\eta|/2$. In that case we put $a^\ast=a$ (thus we
do not adjoin an extra arrow). 
\end{enumerate}
Taking into account Remark \ref{ref-10.2.3-38} we may assume that
$\sigma=\sum_i e_i\otimes e_i$.  This leads to the following
construction. We start with the data $(Q,d,w)$ consisting of (1) an integer $d\ge 3$, (2) a
finite graded quiver $Q$ with arrows of degree $\ge (2-d)/2$, (3) an element $w\in
k\tilde{Q}/[k\tilde{Q},\tilde{Q}]$ 
(with $\tilde{Q}$ as above with $|\eta|=2-d$) satisfying
$\{w,w\}=0$ for the necklace bracket on
$k\tilde{Q}/[k\tilde{Q},k\tilde{Q}]$.

Let $\bar{Q}$ be obtained from $\tilde{Q}$ by adjoining at every vertex $i$
a loop $z_i$ of degree $1-d$. Put $z=\sum_i z_i$. 
Then $\Pi(Q,d,w)$ is the DG-algebra $(k\bar{Q},d)$ with differential
\begin{align}
dz&=\sum_{a\in Q} [a,a^\ast] \label{ref-10.7-40}\\
df&=\{w,f\}\qquad \text{for $f\in k\tilde{Q}$} \label{ref-10.8-41}
\end{align}
By the above discussion we have
\[
\Pi(Q,d,w)\cong\Pi(V_c,\eta,w)
\]
where
\[
\eta=\sum_{a\in Q}[a,a^\ast]
\]
and $V_c$ is the $l$-bimodule corresponding to the quiver $\tilde{Q}$. 
\begin{remarks} The condition \eqref{ref-10.8-41} may be rewritten as
\begin{equation}
\label{ref-10.9-42}
\begin{aligned}
da&=(-1)^{(|a|+1)|a^\ast|} \frac{{}^\circ\partial w}{\partial a^\ast}\\
da^\ast&=-(-1)^{|a|}\frac{{}^\circ\partial w}{\partial a}
\end{aligned}
\end{equation}
for $a\in Q$ where ${}^\circ \partial/\partial x$ for is the \emph{circular 
derivative}
\[
\frac{{}^\circ \partial w}{\partial x}=\sum_{w=uxv}(-1)^{|u|(|x|+|v|)}vu
\]
\end{remarks}
\begin{remarks}
Since we are assuming that $k$ is algebraically closed we can in fact
always reduce to the quiver case up to Morita equivalence. Indeed
we have $l=\bigoplus_i l_i$ for $l_i=M_{p_i}(k)$. Let $e$ be an idempotent
in $l$ such that $ele$ is isomorphic to the center $Z(l)$ of $l$.  Clearly
$\Pi(V_c,\eta,w)$ is Morita equivalent to $e\Pi(V_c,\eta,w)e$. 

Put
$V'_c=eV_ce$. Then $(V_c\otimes_l V_c)_l\cong (V'_c\otimes_{Z(l)} V'_c)_{Z(l)}$.
So $\eta$ corresponds to an element $\eta'$ of $(V'_c\otimes_{Z(l)} V'_c)_{Z(l)}$.
Put $w'=ewe$. One verifies easily that $\{w',w'\}_{\omega_{\eta'}}=0$ and
\[
e\Pi(V_c,\eta,w)e\cong\Pi(V'_c,\eta',w')
\]
\end{remarks}

\section{Proof of Theorem \ref{ref-10.2.2-37}}
\label{ref-11-43}
As the title of this section indicates we will prove Theorem
\ref{ref-10.2.2-37}.  The proof consists of a number of steps. In
\S\ref{ref-11.1-44} we will translate the non-degeneracy
condition for Hochschild cycles into a more tractable form.  In
\S\ref{ref-11.2-57} we obtain a first classification of exact Calabi-Yau
algebras. Finally in \S\ref{ref-11.3-68} we complete the proof. 

\subsection{Non-degeneracy in the cofibrant case}
\label{ref-11.1-44}
In this section we assume that $A=(T_lV,d)$ is a
cofibrant object in $\PCAlgc(l)$ (see \S\ref{ref-A.4-97}).  We
write the differential as $d=d_1+d_2+\cdots$ as in
\S\ref{ref-A.5-99}.

By the pseudo-compact analogue of Proposition \ref{ref-7.2.1-15}  (see the end of
 \S\ref{ref-7.2-12}) we have $\C(A)\cong \cone((\Omega^1_A)_{\natural}\xrightarrow{\partial_1} A_l)$.
An element $\xi$ of $\HH_d(A)$ is now represented by a pair
$(\omega,a)$ where $\omega\in (\Omega_l^1 A)_{\natural}$, $a\in A_l$,
$d\omega=0$, $\partial_1 \omega=da$, $|\omega|=d-1$, $|a|=d$. We would
like to know when $\xi$ is non-degenerate. We answer this question in a typical
case in Lemma \ref{ref-11.1.2-46} below.

We have a morphism of $X$-complexes
\[
\xymatrix{
A_l \ar[r]^{\partial_0} \ar[d]&(\Omega_l^1 A)_{\natural} \ar[d]_{\text{res}}\ar[r]^{\partial_1}& A_l\ar[d]\\
l_l\ar[r]_0& V_l\ar[r]_0 & l_l
}
\]
The outermost maps send $f\in A$ to its image in $\bar{f}$ in $l$. 
The middle map sends $f Dv$ for $v\in V$ to $\bar{f}v$.  We obtain a corresponding map
\begin{equation}
\label{ref-11.1-45}
\C(A)\r ((l\oplus \Sigma V)_l,d_1)
\end{equation}
of complexes.
\begin{remarks} It is easy to show that that \eqref{ref-11.1-45} can be described
intrinsically as being obtained from the the standard map
\[
\C(A)=A\Lotimes_{A^e} A\r (l\Lotimes_A l)_l
\]
We present it in the above explicit way since that is how we will use it. 
\end{remarks}
\begin{lemmas} \label{ref-11.1.2-46}
Assume $V$ is of the following form
\begin{enumerate}
\item \label{ref-1-47}$d_1:V\r V$ is zero.
\item \label{ref-2-48}$
V=V_c\oplus lz
$
with $z$ an $l$-central element of degree $-d+1$ and $V_c$ finite dimensional. 
\item \label{ref-3-49} $V_c$ is concentrated in degrees $[-d+2,0]$. 
\item \label{ref-4-50} $dz=\sigma'\eta \sigma''\mod V^{\otimes_l 3}$ with $\eta\in (V_c\otimes_l
  V_c)_l$ being a non-degenerate  element
  of degree $-d+2$.
\end{enumerate}
Then $(\omega,a)$ is non degenerate if and and only if $\res\omega=uz^\dagger$
for $u$ a central unit in~$l$ (see \eqref{ref-5.1-4} for the notation $(-)^\dagger$). 
\end{lemmas}
We use the following preparatory lemma.
\begin{lemmas} \label{ref-11.1.3-51} Let $M$, $N$ be objects in  $\PCDGMod(A^e)$ which
  are of the form $(A\otimes_l M_0\otimes_l A, d)$, $(A\otimes_l
  N_0\otimes_l A, d)$, such that $l^e\otimes _{A^e} M=M_0$, $l^e\otimes_{A^e} N=N_0$ are
finite dimensional and have
zero differential. Let $\xi\in H^{u}(M\otimes_{A^e} N)$. The following are equivalent
\begin{enumerate}
\item $\xi$ is non-degenerate.
\item The image
$\bar{\xi}$ of $\xi$ in $H^u(M_0\otimes_{l^e} N_0)$ is non-degenerate. 
\end{enumerate}
\end{lemmas}
\begin{proof} We prove (1)$\Rightarrow$(2). If $\xi$ is non-degenerate then
it produces an isomorphism in $D(A^e)$
\begin{equation}
\label{ref-11.2-52}
\xi^+:\Hom_{A^e}(M,A\otimes A)\r N
\end{equation}
Both sides are cofibrant, left and right tensoring with $l$ gives us an isomorphism
in $D(l^e)$
\begin{equation}
\label{ref-11.3-53}
\Hom_{l^e}(M_0,l\otimes l)\r N_0
\end{equation}
which is easily shown to equal to $\bar{\xi}^+$. Hence $\bar{\xi}$ is non-degenerate.

Now we prove (2)$\Rightarrow$(1). Since $\bar{\xi}$ is non-degenerate and
since $M_0$, $N_0$ have zero differential we see that $\bar{\xi}$
induces an isomorphism  $\Hom_{l^e}(M_0,l\otimes l)\cong N_0$
as $l$-bimodules.
It follows immediately from Nakayama's lemma applied to $A^e$ (without
differential) that   $\xi^+:\Hom_{A^e}(M,A\otimes A)\r N$ is an isomorphism
of graded $A$-modules. As this isomorphism is compatible with the differential
it is an isomorphism of DG-bimodules. 
\end{proof}
\begin{proof}[Proof of Lemma \ref{ref-11.1.2-46}]
We will first prove the $\Leftarrow$ direction. Thus if $\xi=(\omega,a)\in \HH_d(A)$
and $\res\omega=uz^\dagger$ for $u$ a central unit in $l$ then we must prove that 
$\xi$ is non-degenerate. Replacing $z$ by $uz$ we may assume $u=1$.

We will view $\xi$ as a cohomology class of degree $-d$ in the total
complex of the triple complex obtained from tensoring the resolution
\eqref{ref-7.4-16} with itself over $A^e$.  This triple complex looks
as follows
\begin{equation}
\label{ref-11.4-54-1}
  \xymatrix{
&0& 0&\\
0\ar[r] &   \Omega_l^1 A\otimes_{A^e}(A\otimes_l A) \ar[r]^{\partial_1\otimes 1}\ar[u] & (A\otimes_l A)\otimes_{A^e} (A\otimes_l A)\ar[r]\ar[u]&0\\
0\ar[r] &   \Omega_l^1 A\otimes_{A^e} \Omega_l^1 A\ar[r]_{\partial_1\otimes 1}\ar[u]^{1\otimes \partial_1}& (A\otimes_l A)\otimes_{A^e} \Omega_l^1 A
\ar[u]_{1\otimes \partial_1}\ar[r]&0\\
&0\ar[u] & 0\ar[u]&
  }
\end{equation}
Or simplified
\begin{equation}
\label{ref-11.4-54}
\xymatrix{
  &0 & 0\\
  0\ar[r]&(\Omega_l^1 A)_l\ar[r]^{\partial^{\text{hor}}_1} \ar[u] & (A\otimes_l A)_l\ar[u]\ar[r]& 0\\
  0\ar[r]&\Omega_l^1 A\otimes_{A^e}\Omega_l^1 A
  \ar[u]^{1\otimes \partial_1}\ar[r]_-{\partial_1\otimes 1} & (\Omega_l^1 A)_l\ar[r]\ar[u]_{\partial^{\text{ver}}_1} & 0\\
  & 0\ar[u] & 0\ar[u] }
\end{equation}
where the only ambiguous identification we have used is 
\begin{equation}
\label{ref-11.5-55}
(A\otimes_l A)\otimes_{A^e} (A\otimes_l A)\r (A\otimes_l A)_l: (a\otimes b)\otimes (c\otimes d)
\mapsto (-1)^{|d|(|a|+|b|+|c|)} da\otimes bc
\end{equation}
For further reference we denote this triple complex by $Y(A)$. 

$\xi$ is now represented by a sum $\sum_i s\omega_i\otimes
s\omega'_i+s\psi+s\psi'+c$ where $\omega_i,\omega'_i\in \Omega_l^1 A$, $\psi,\psi'\in
(\Omega_l^1 A)_l$, $c\in (A\otimes_l A)_l$.

From the condition $(d+\partial_1\otimes
1+1\otimes \partial_1)(\xi)=0$ we deduce the following conditions in
$\Sigma(\Omega_l^1 A)_l$ (viewed respectively as $(A\otimes_l
A)\otimes_{A^e} \Sigma\Omega_l^1 A$ and $\Sigma\Omega_l^1 A\otimes_{A^e}
(A\otimes_l A)$).
\begin{align*}
-sd\psi+\sum_i \partial_1 \omega_i\otimes s\omega'_i&=0\\
-sd\psi'+\sum_i (-1)^{|\omega_i|+1}s\omega_i\otimes \partial_1\omega'_i&=0\\
\end{align*}
which in $(\Omega_l^1 A)_l$ becomes
\begin{equation}
\label{ref-11.6-56}
\begin{aligned}
d\psi&=\sum_i (-1)^{|\omega_i|}\partial_1 \omega_i\otimes \omega'_i\\
d\psi'&=\sum_i (-1)^{|\omega_i|+1}\omega_i\otimes \partial_1\omega'_i\\
\end{aligned}
\end{equation}
We consider these identities modulo $V^2\cdot DV+V\cdot DV \cdot
V+DV\cdot V^2$ in $\Omega_l^1 A$.

Keeping only the terms that can contribute (using the fact that $|\xi|=d$, and
the fact that $d_1=0$) we get
\begin{align*}
\psi&=Dz^\dagger+\cdots\\
\psi'&=Dw+\cdots\\
\sum_i\omega_i\otimes \omega'_i&=\sum_j Du_j\otimes Dv_j+\cdots
\end{align*}
for $w,u_j,v_j\in V$.

Translating \eqref{ref-11.6-56} we get in $(\Omega_l^1 A)_l$.
\begin{align*}
D(dz^\dagger)_2&=\sum_j (-1)^{|u_j|} ((-1)^{|u_j||v_j|}(Dv_j)u_j-u_j(Dv_j))\\
D(dw)_2&=\sum_j (-1)^{|u_j|+1}( (Du_j)v_j-(-1)^{|u_j||v_j|} v_j(Du_j))
\end{align*}
where $(-)_2$ denotes the quadratic part. 

Projecting on $(DV.V)_l\cong (V\otimes_l V)_l$ and $(V\cdot DV)_l\cong (V\otimes_l V)_l$ we get
\begin{align*}
(dz^\dagger)'_2\otimes (dz^\dagger)''_2&=\sum_j (-1)^{|u_j|} (-1)^{|u_j||v_j|}v_j\otimes u_j\\
(dz^\dagger)'_2\otimes (dz^\dagger)''_2&=\sum_j (-1)^{|u_j|+1}u_j\otimes v_j \\
(dw)'_2\otimes (dw)''_2&=\sum_j (-1)^{|u_j|+1}u_j\otimes v_j\\
(dw)'_2\otimes (dw)''_2&=\sum_j (-1)^{|u_j|}(-1)^{|u_j||v_j|}
v_j\otimes u_j
\end{align*}
We conclude $dz_2^\dagger=dw_2\quad\mod [l,-]$ and furthermore 
\[
\sum_j (-1)^{|u_j|(|v_j|+1)}v_j\otimes u_j=\eta_2\qquad  \mod [l,-]
\]
We will apply Lemma \ref{ref-11.1.3-51} to $Y(A)$. We find that  $\bar{\xi}$  is of the form
\[
sDz\otimes 1+1\otimes sDz+(-1)^{|\eta'|(|\eta''|+1)} sD\eta'\otimes sD\eta''
\]
By the non-degeneracy of $\eta$, this is clearly a non-degenerate
element of $(l\oplus \Sigma V)\otimes (l\oplus\Sigma V)$. Hence we
are done.

\medskip

Now we prove the ``easy'' direction $\Rightarrow$. Let $\xi=\sum_i s\omega_i\otimes
s\omega'_i+s\psi+s\psi'+c$ in \eqref{ref-11.4-54} represent a non-degenerate
element in $\HH_d(A)$. Then 
\[
\bar{\xi}=\sum_j su_j\otimes s v_j+sw\otimes 1+1\otimes sw'
\]
for $u_j, v_j, w,w'\in V$ where we have identified $l^e\otimes_{A^e} \Omega_l^1 A$ with $(DV)_l\cong V_l$. 
This element must be non-degenerate in
$((\Sigma V\oplus l)\otimes (\Sigma V\oplus l))_l$ by Lemma \ref{ref-11.1.3-51}. 
This can only happen if
$w\in V_l$ is a generator for $(V_{-d+1})_l$. Then $w$ must
necessarily be of the form $uz^\dagger$ for an invertible element $u\in
Z(l)$.
\end{proof}

\subsection{First classification of exact Calabi-Yau algebras}
\label{ref-11.2-57}
\begin{theorems} \label{ref-11.2.1-58} Assume the characteristic of $k$ is zero. Let $A\in
  \PCAlgc(l)$. Assume that $A$ is concentrated in degrees $\le 0$ and let $d\ge 3$.
Then the following are equivalent. 
\begin{enumerate}
\item \label{ref-1-59} $A$ is exact $d$-Calabi-Yau. 
\item \label{ref-2-60} There is a weak equivalence $(T_l V,d)\r A$
as augmented $l$-DG-algebras with~$V$ having the following properties
\begin{enumerate}
\item \label{ref-2a-61}$d_1:V\r V$ is zero.
\item \label{ref-2b-62}$
V=V_c\oplus lz
$
with $z$ an $l$-central element of degree $-d+1$ and $V_c$ finite dimensional. 
\item \label{ref-2c-63} $V_c$ is concentrated in degrees $[-d+2,0]$.
\item \label{ref-2d-64} $dz=\sigma'\eta \sigma''$ with $\eta\in (V_c\otimes_l
  V_c)_l$ being a non-degenerate and anti-symmetric 
element under $\ss$ (cfr. \S\ref{ref-5.2-5}).
\end{enumerate}
\item \label{ref-3-65} There is a weak equivalence $(T_l V,d)\r A$
as augmented $l$-DG-algebras with~$V$ as in (\ref{ref-2a-61}-\ref{ref-2c-63}) and (\ref{ref-2d-64}) is replaced by
\begin{enumerate}
\item[(d')]
$
dz=\sigma'\eta \sigma''=\sigma'\eta_2\sigma''+\sigma'\eta_3 \sigma''+\cdots
$
such that $\eta$ is a sum of commutators in $T_lV_c$, $\eta_n\in (V_c^{\otimes_l n})_l$ and $\eta_2
\in (V_c\otimes_l V_c)_l$ is non-degenerate.
\end{enumerate}
\end{enumerate}
\end{theorems}

\begin{proof}
  We first prove the direction (\ref{ref-1-59})$\Rightarrow$(\ref{ref-3-65}).
  We employ Koszul duality (see \S\ref{ref-A.5-99}). Since $A$
  is exact $d$-Calabi-Yau it is in particular $d$-Calabi-Yau and
  hence the algebra $\Ext^*_A(l,l)$ is finite dimensional and
  symmetric with an invariant form of degree $d$ (see \cite[Lemma
  3.4]{Keller11}). Thus by Corollary \ref{ref-A.5.2-102}
  $H^*(A^!)$ is symmetric with an invariant form of degree~$d$. From
  the construction of $A^!$ (cfr. \eqref{ref-A.9-100}) it follows
  immediately $\bar{A^!}$ is concentrated in degrees $\ge 1$. Hence
  the same holds for $H^\ast(A)$.  Thus $H^\ast(A^!)=l\oplus W_c\oplus
  lh$ with $h$ being $l$-central of degree $d$ and $W_c$ being
  concentrated in degrees $[1,d-1]$. The $h$ element is constructed as the
  orthogonal to the augmentation.

By Proposition
  \ref{ref-A.5.4-103} this implies that $A$ has a minimal model $(T_l
  V,d)$ with $V=V_c\oplus lz$ as in Lemma \ref{ref-11.1.2-46} (with $\eta$
  in addition being anti-symmetric under $\ss$).

By Corollary \ref{ref-D.4-119} we may replace $A$ by $(T_l V,d)$.  
Assuming now $A=(T_lV,d)$ we note that in
  characteristic zero the third quadrant double complex
\[
\cdots \xrightarrow{\partial_1} (A/l)_l \xrightarrow{\partial_0} (\Omega^1_l A)_\natural \xrightarrow{\partial_1} (A/l)_l\r (A/l+[A,A])_l\r 0
\]
has exact rows (see \cite[Thm 3.1.4]{Loday1} and proof). In other words
the reduced cyclic homology of $A$ is equal to the homology of $(A/l+[A,A])_l$. The 
Connes long exact sequence for (reduced) cyclic homology is obtained from the
distinguished triangle
\[
\Sigma (A/l+[A,A])_l\xrightarrow[B]{\partial_0} \cone ((\Omega^1_l A)_\natural \xrightarrow{\partial_1} (A/l)_l)\xrightarrow[I]{} (A/l+[A,A])_l\r 
\]

Let $\xi\in \HC_{d-1}(A)$ be such that $B\xi$ is non-degenerate. Then
$\xi$ corresponds to an element $\bar{\chi}$ of degree $-d+1$ in
$(A/l+[A,A])_l$ with $d\bar{\chi}=0$. In other words
\begin{equation}
\label{ref-11.7-66}
d\chi=\sum_i [x_i,y_i]\qquad \mod [l,-]
\end{equation}
for $x_i,y_i\in \bar{A}$.  We put $\eta'=\sum_i [x_i,y_i]$.

The element $\bar{\chi}$ is sent under
$\partial_0$ to $(D\chi,0)$.  Since $(D\chi,0)$ is a non-degenerate
element of Hochschild homology it follows from Lemma \ref{ref-11.1.2-46}
that $\chi=uz^\dagger+v$ for $u$ an invertible central element in $l$
and $v\in \bar{A}^2$. Put $z'=\sigma' \chi\sigma''$ and
$V'=V_c+lz'$. From \eqref{ref-11.7-66} we deduce 
\[
dz'=\sigma'\eta' \sigma''=\sigma'\eta'_2\sigma''+\sigma'\eta'_3 \sigma''+\cdots
\]
such that $\eta'$ is a sum of commutators in $T_lV_c$, $\eta'_n\in (V_c^{\otimes_l n})_l$ and $\eta'_2
\in (V_c\otimes_l V_c)_l$. As $T_lV=T_lV'$ we obtain from Corollary \ref{ref-A.5.6-105} that
 $V'=\Sigma^{-1}\DD \overline{A^!}$
 and hence $\eta'_2$ is still non-degenerate.

\medskip

Now we prove (\ref{ref-3-65})$\Rightarrow$(\ref{ref-2-60}). This
is a version of  \cite[Prop.\ 10.1.2]{KS2}. We first note that the condition $dz=\sigma'\eta\sigma''$ 
with $\eta$ being a sum of
commutators in $T_l V_c$
is obviously invariant under isomorphisms $q:(T_lV,d)\r (T_lV,d')$ 
of the form $q(v)=v+\text{higher terms}$ with $q(z)=z$ and $q(T_l V_c)\subset
T_l V_c$ (this last condition is in fact automatic for degree reasons). 

Assume that we have shown that $A$ is weakly equivalent to $(T_l V,d)$
such that $dz=\sigma' \eta\sigma''$ with $\eta_3=\cdots=\eta_{n-1}=0$. We
will construct an isomorphism $q:(T_l V,d)\r (T_l V,d')$ of augmented
pseudo-compact $l$-DG-algebra of the form $q(v)=v+\beta(v)$ for
$\beta(v)\in V_c^{\otimes_l n-1}$ such that $\beta(z)=0$ and 
such that $d'z=\sigma'\eta'\sigma''$ with
$\eta'_{n}=0$ (and of course $\eta_i=\eta'_i$ for $i\le n-1$). Repeating
this procedure we kill in the limit all the higher order terms of
$dz$. 

We have $d'=q\circ d\circ q^{-1}$. 
\begin{align*}
d'z&=qdz\\
&=\sigma' q(\eta)\sigma''\\
&=\sigma' \eta_2\sigma''+ \sigma'\eta_n\sigma''+\sigma'\beta(\eta_2')\eta_2''\sigma''+\sigma'\eta_2'\beta(\eta''_2)\sigma''+\cdots
\end{align*}
Thus we must solve the following equation in $(V^{\otimes_l n}_c)_l$
\[
\eta_n+\beta(\eta_2')\eta_2''+\eta_2'\beta(\eta''_2)=0
\]
for an $l$-bimodule map $\beta:V_c\r V^{\otimes_l n-1}_c$. This can be rewritten as
\begin{equation}
\label{ref-11.8-67}
\begin{aligned}
0&=\eta_n-(-1)^{|\eta_2'||\eta''_2|}\beta(\eta_2'')\eta_2'+\eta_2'\beta(\eta''_2)\\
&=\eta_n-[\eta_2',\beta(\eta_2'')]
\end{aligned}
\end{equation}
As \eqref{ref-11.8-67} is a linear algebra problem we may without
loss of generality assume that $k$
is algebraically closed. Thus $l=\bigoplus_i l_i$ for $l_i=M_{p_i}(k)$. 

It is furthermore easy to see that \eqref{ref-11.8-67} is invariant under
Morita equivalence. Therefore we may replace $l$ by its center and so
we are reduced to $l=\sum_{i=1}^m ke_i$ for central orthogonal idempotents $(e_i)_i$.
I.e.\ the ``quiver case''.

As in \S\ref{ref-10.3-39} we may bring $\eta_2$ in the following form
\[
\eta_2=\sum_{a\in Q}[a,a^\ast]
\]
where $Q$ is a suitable graded quiver with vertices $\{1,\ldots,m\}$.
Now one verifies that~$\eta_n$, being a sum of commutators can be written
as
\[
\eta_n=\sum_{a\in Q}  [a,\eta_a]+[a^\ast,\eta_{a^\ast}]
\qquad \text{(modulo
$[l,-]$)}
\]
for certain paths $\eta_a$, $\eta_{a^\ast}$ of length $n-1$ in
$\tilde{Q}$ (the quiver corresponding to the $l$-bimodule $V_c$, see
\S\ref{ref-10.3-39}). If $a=a^\ast$ then we may and we will assume $\eta_a=
\eta_{a^\ast}$. 
It now suffices to define
\begin{align*}
\beta(a)&=-(-1)^{|a||a^\ast|} \eta_{a^\ast}\\
\beta(a^{\ast})&=\eta_a
\end{align*}
to obtain the solution to \eqref{ref-11.8-67}. Note that if $a=a^\ast$ then
$|a|=|a^\ast|$ is odd and hence $-(-1)^{|a||a^\ast|}=1$, as it should.  

\medskip

Finally  we prove the direction (\ref{ref-2-60})$\Rightarrow$(\ref{ref-1-59}). We consider $\bar{z}^\dagger$
as an element of $\HC^{\text{red}}_{d-1}(A)=H^{-d+1}(A/(l+[A,A]))_l$ as we have
indeed $d\bar{z}^\dagger=0$. Then $\partial_0 \bar{z}^\dagger=(Dz^\dagger,0)$. Since $\res Dz^\dagger=z^\dagger$ it is now sufficient to invoke Lemma \ref{ref-11.1.2-46}. 
\end{proof}

\subsection{Completion of proof}
\label{ref-11.3-68}
Theorem \ref{ref-10.2.2-37} now follows from the equivalence of (\ref{ref-1-59}) and (\ref{ref-2-60}) in 
Theorem \ref{ref-11.2.1-58} together with the following lemma. 
\begin{lemmas}
\label{ls}
 Assume that $k$ has characteristic zero. Let $A=(T_lV,d)$ be an augmented pseudo-compact $l$-DG-algebra
with the following properties for $d\ge 3$:
\begin{enumerate}
\item $V=V_c\oplus lz$
with $z$ an $l$-central element of degree $-d+1$ and $V_c$ finite dimensional.
\item $dz=\sigma'\eta \sigma''$ with $\eta\in (V_c\otimes_l
  V_c)_l$ being non-degenerate and anti-symmetric 
under $\ss$.
\item $T_l V_c$ is stable under $d$. 
\end{enumerate}
Then there is some $w\in T_l V/[T_l V,T_l V]$ 
with $|w|=|\eta|+1=-d+3$ and $\{w,w\}_{\omega_\eta}=0$
such that for $f\in T_l V_c$ we have
\[
df=\{w,f\}_{\omega_\eta}
\]
Hence $A=\Pi(V_c,\eta,w)$. Furthermore if $d_1:V\r V$ is zero then $w$
will have only cubic terms and higher.
\end{lemmas}
\begin{proof}
This can be deduced from the general machinery of non-commutative symplectic
geometry but we will give an explicit proof. Since $d^2=0$ we obtain
\[
\sigma' ((d\eta')\eta'' + (-1)^{|\eta'|}\eta' (d\eta''))\sigma''=0
\]
and hence
\[
(d\eta')\eta'' + (-1)^{|\eta'|}\eta' (d\eta'')=0\qquad \text{mod $[l,-]$}
\]
This can be rewritten as (everything mod $[l,-]$)
\begin{align*}
(-1)^{|\eta'|}\eta' (d\eta'')&=-(d\eta')\eta''\\
&=(-1)^{|\eta'||\eta''|} (d\eta'')\eta'
\end{align*}
where the last identity follows from applying $d\otimes 1$ to
$\eta'\otimes \eta''=-(-1)^{|\eta'||\eta''|}\eta''\otimes \eta'$. 
In other words 
\begin{equation}
\label{ref-11.9-69}
\bar{w}\overset{\text{def}}{=} (-1)^{|\eta'|+1}\eta' d\eta''
\end{equation} 
is
a cyclically symmetric element of $(T_lV_c)_l$ of
degree $-d+3$. 

If $\phi\in V^D_c$ then we define
\[
\partial_\phi: (T_l V_c)_l
\r T_l V_c:
a_1\otimes\cdots \otimes a_n\r \phi(a_1)''a_2\otimes\cdots\otimes a_n \phi(a_1)'
\]
(there is no sign here since $\phi(a_1)'$ is a scalar).

The element $\eta\in  (V_c\otimes_l V_c)_l$ of degree $-d+2$ defines a map
\[
\eta^+:V_c^D\r V_c: \phi\mapsto (-1)^{|\phi||\eta|} 
\phi(\eta')''\eta''\phi(\eta')'
\]
of degree $-d+2$ (again the fact that $\phi(\eta')'$ is a scalar makes the sign
rather trivial). Using the definition of $\bar{w}$ we obtain the following identity
\[
d(\eta^+(\phi))=(-1)^{|\phi||\eta|} (-1)^{|\phi|+1}\partial_\phi(\bar{w})
\]
(where we use that for non-zero terms we have $|\phi|+|\eta'|=0$). 

Let $w$ be an  inverse image of $\bar{w}$ under the cyclic symmetrization map
\[
(T_l V_c)_l\r (T_l V_c)_l:
a_1\otimes\cdots \otimes a_n\r \sum_i \pm a_i\otimes \cdots \otimes a_n\otimes
a_1\otimes\cdots \otimes a_{i-1}
\]
E.g.\ one possible choice is
\[
w=\sum_n \frac{\bar{w}_n}{n}
\]
where $\bar{w}_n\in (V_c^{\otimes_l n})_l$. 

If $\phi\in V_c^D$ then we have an associated  double $l$-derivation
\[
i_\phi: T_lV_c\r T_lV_c\otimes T_lV_c
\]
which sends $v$ to $\phi(v)$. We get an induced map
\[
\iota_\phi:T_lV_c/[T_l V_c,T_l V_c]\r T_l V_c:\bar{f}\mapsto (-1)^{|i_\phi(f)''||i_\phi(f)'|} 
i_\phi(f)''i_\phi(f)'
\]
which sends $a_1\otimes\cdots \otimes a_n$ to
\[
\sum_i \pm \phi(a_i)''a_{i+1}\otimes \cdots \otimes a_n
\otimes a_1\otimes \cdots \otimes a_{i-1}\phi(a_i)'
\]
It follows that
\[
\iota_\phi(w)=\partial_\phi(\bar{w})
\]
and thus
\begin{equation}
\label{ref-11.10-70}
d(\eta^+(\phi))=(-1)^{|\phi||\eta|}(-1)^{|\phi|+1}\iota_\phi(w)
\end{equation}
Since $\eta$ is non-degenerate we have an inverse to $\eta^{+}$
\[
\eta^-:V_c\r V_c^D
\]
of degree $d-2$. Applying \eqref{ref-11.10-70} with   $\phi=\eta^{-}(v)$ 
for $v\in V_c$ yields
\begin{equation}
\label{ref-11.11-71}
dv=(-1)^{(|v|+d-2)(d-2)+(|v|+d-2)+1}\iota_{\eta^-(v)}(w)=(-1)^{|v|(d+1)+1}\iota_{\eta^-(v)}(w)
\end{equation}
We have
\begin{align*}
v=\eta^+(\phi)&=
(-)^{|\phi||\eta|}\phi(\eta')''\eta''\phi(\eta')'\\
&=-(-)^{|\phi||\eta|}(-1)^{|\eta'||\eta''|}\phi(\eta'')''\eta'\phi(\eta'')'
\end{align*}
Hence
\begin{align*}
2\iota_{\phi}\omega_\eta&=\iota_{\phi}D\eta'D\eta''\\
&=\phi(\eta')'' 
(D\eta'') \phi(\eta')'-(-1)^{|\eta'||\phi|} \phi(\eta'')'' (D\eta') \phi(\eta'')'\\
&=2(-1)^{|\phi||\eta|}Dv
\end{align*}
So ultimately we find
\[
\iota_{\eta^{-}(v)}(\omega_\eta)=(-1)^{(|v|+d-2)(d-2)}Dv=(-1)^{d(|v|+1|)}Dv
\]
and hence
\[
\eta^{-}(v)=(-1)^{d(|v|+1|)}H_v
\]
Comparing with \eqref{ref-11.11-71} 
we find
\begin{align*}
dv&=(-1)^{|v|(d+1)+1}(-1)^{d(|v|+1|)}\iota_{H_v}(w)\\
&=(-1)^{|v|+d+1} \iota_{H_v}(w)
\end{align*}
Thus we get
\begin{align*}
  dv&=(-1)^{|v|+d+1} \iota_{H_v} w\\
&=(-1)^{|v|+d+1}\iota_{H_v} Dw\\
  &=-(-1)^{|v|+d+1} (-1)^{(|v|+d-2)(|w|+d-2)}\ldb w,v\rdb'_{\omega_\eta}\ldb w,v\rdb''_{\omega_\eta}\\
  &=-(-1)^{|v|+d+1} (-1)^{|v|+d}\{w,v\}_{\omega_\eta}\\
  &=\{w,v\}_{\omega_\eta}\\
\end{align*}
(we refer  to \eqref{ref-10.1-29} and \eqref{ref-10.4-33} for the sign in the third line). 
Since $\{w,-\}_{\omega_\eta}$ is a derivation in its second argument we finally obtain
for $f\in T_l V_c$
\[
df=\{w,f\}_{\omega_\eta}
\]
We must prove $\{w,w\}_{\omega_\eta}=0$. Since $d^2=0$ we obtain as in \eqref{ref-10.6-36}
that $\{\{w,w\}_{\omega_\eta},v\}_{\omega_\eta}=0$ for all $v \in V$.  Thus we must prove
for $u\in T_l V_c$
\[
\forall v\in V:\{u,v\}_{\omega_\eta}=0\Rightarrow u=0
\]
This is a linear statement so we may assume $k$ is algebraically closed. It is also
easy to see that it is invariant under Morita equivalence so we may pass to the
quiver case. Then the statement follows immediately from a similar expression as \eqref{ref-10.9-42}.

Finally it remains to show that $w$ contains only cubic terms and higher if $d_1=0$. This
follows immediately from the explicit formula \eqref{ref-11.9-69}.
\end{proof}

\section{Exact Calabi-Yau algebras and cyclic $A_\infty$-algebras}
\label{ref-12-72}
In this section we assume that $k$ has characteristic zero. Let $A$ be a finite dimensional $l$-$A_\infty$-algebra. An $A_\infty$-\emph{cyclic structure} 
of degree $d$ on $A$ is a symmetric bilinear form 
\[
(-,-) :A\times A\r \Sigma^d A
\]
of degree $d$ such that 
\begin{equation}
\label{ref-12.1-73}
(m_n(a_1,\ldots,a_n),a_{n+1})=(-1)^n(-1)^{|a_1|(|a_2|+\cdots+|a_{n+1}|)}  (m_n(a_2,\ldots,a_{n+1}),a_1)
\end{equation}
The following result can be used as an alternative approach to  Lemma \ref{ls}
which is part of the proof of 
Theorem \ref{ref-10.2.2-37} 
(see e.g.\ \cite{Lazaroiu} for the relation). 
\begin{theorem} \label{ref-12.1-74} Assume that $k$ has characteristic
  zero. Let $A\in \PCAlgc(l)$ be homologically smooth and assume that
  the grading on $A$ is concentrated in degrees $\le 0$.

Then the following statements are equivalent
\begin{enumerate}
\item $A^!$ has a finite dimensional minimal $A_\infty$-model (as
  augmented $l$-$A_\infty$-algebra, see \ref{ref-A.3-93}) with a
  cyclic $A_\infty$-structure of degree $d$.
\item $A$ is exact $d$-Calabi-Yau.
\end{enumerate}
\end{theorem}
We will give the proof under the technical assumption $d\ge 3$. We leave the obvious modifications for $0\le d<3$ to the reader.
We use the following technical lemma.
\begin{lemma} \label{ref-12.2-75}
Let $W=l\oplus W_c\oplus lh$ be a finite dimensional minimal augmented
  $l$-$A_\infty$-algebra with $h$ being $l$-central. Define $(-,-)$ as the composition
\[
(-,-):W\times W\xrightarrow{m_2} W\xrightarrow{\text{projection}} lh\cong l
\xrightarrow{\Tr} k
\]
Then $(-,-)$ defines a cyclic $A_\infty$ structure if and only if 
\begin{enumerate}
\item  $(-,-)$ is a non-degenerate symmetric form.
\item  The $m_n$ for $n\ge 3$ have their image in $W_c$.
\end{enumerate}
\end{lemma}
\begin{proof} We first prove the $\Rightarrow$-direction. The non-degeneracy
of $(-,-)$ and the fact that it is symmetric is by definition. Changing 
$h$ by a non-zero scalar we may assume $(h,1)=(1,h)=1$.

The
cyclic condition \eqref{ref-12.1-73} for $n\ge 3$ gives for $u\in l$
\[
(m_n(a_1,\ldots,a_n),u)=\pm (m_n(u,a_1,\ldots,a_{n-1}),a_n)=0
\]
Thus $m_n(a_1,\ldots,a_n)$ must indeed be contained in $W_c$. 

Now we prove the $\Leftarrow$-direction. We must prove \eqref{ref-12.1-73}
which simplifies to 
\[
(m_n(a_1,\ldots,a_n),a_{n+1})=(-1)^n(-1)^{|a_1|(2-n)} (a_1,m_n(a_2,\ldots,a_{n+1}))
\]
We write out the $A_\infty$-axiom for the $m$'s (see e.g.\ \cite{Keller}), retaining only the terms
which have a non-zero projection on $lh$. This yields
\[
(-1)^nm_2(m_n(a_1,\ldots,a_n),a_{n+1})-(-1)^{|a_1|(2-n)}m_2(a_1,m_n(a_2,\ldots,a_{n+1}))+\cdots
=0
\]
Taking the projection on $lh$ gives what we want. 
\end{proof}

\begin{proof}[Proof of Theorem \ref{ref-12.1-74}]  We first prove (2)$\Rightarrow$(1). 
Thanks to Theorem \ref{ref-11.2.1-58} we know that $A$ is weakly equivalent to
$(T_l V,d)$ where $V$ is as follows.
\begin{enumerate}
\item \label{ref-1-76}$d_1:V\r V$ is zero.
\item \label{ref-2-77}$
V=V_c\oplus lz
$
with $z$ an $l$-central element of degree $-d+1$ and $V_c$ finite dimensional. 
\item \label{ref-3-78} $V_c$ is concentrated in degrees $[-d+2,0]$.
\item \label{ref-4-79} $dz=\sigma'\eta \sigma''$ with $\eta\in (V_c\otimes_l
  V_c)_l$ being a non-degenerate and anti-symmetric 
element under $\ss$ (cfr. \S\ref{ref-5.2-5}).
\end{enumerate}
 It follows from
  Proposition \ref{ref-A.5.4-103} that $W=l\oplus \Sigma^{-1} \DD V=l\oplus
  W_c\oplus lh$ is isomorphic to $A^!$
  as $A_\infty$-algebra where $h=s^{-1} z^\ast$ and $W_c=\Sigma^{-1}\DD V_c$ . 

  We define a symmetric $l$-bilinear form $(W\otimes_l W)_l\r k$ as
  follows: $(1,h)=(h,1)=1$, $(-,-)$ restricted to $W_c\times W_c$ is given by 
contraction with  $\eta$ (in the sense of \eqref{ref-5.3-6}). All other values are zero.

By definition the $(m_n)_n$, restricted to $\bar{W}=W_c\oplus lh$ are dual to the components
$(d_n)_n$ of the differential on $T_lV$. So we deduce from (\ref{ref-4-79}) that
the $m_n$ for $n\ge 3$ have their image in $W_c$. Furthermore
the composition
\[
W_c \times W_c\xrightarrow{m_2}W\xrightarrow{\text{projection}} lh\cong l
\xrightarrow{\Tr} k
\]
is the bilinear form $(-,-)$.  It now suffices to apply Proposition \ref{ref-12.2-75}.

\medskip

Now we prove the implication (1)$\Rightarrow$(2). It is an almost exact inversion of the above
arguments.
Let $W$ be the augmented cyclic minimal model for $A^!$
and let $(-,-)$ be the associated
symmetric non-degenerate $l$-bilinear form  $(W\otimes_l W)_l\r k$ of degree~$d$.
We may write $\bar{W}=W_c\oplus lh$ where $lh$ is dual
to the augmentation. Thus $h$ is $l$-central and we may assume $(h,1)=(1,h)=1$.

We deduce that $(-,-)$ restricts to a symmetric
non-degenerate bilinear form on $W_c$.
All other evaluations of $(-,-)$ on $W=l\oplus W_c\oplus lh$ are zero. 
For $m_2$ we find
\[
(m_2(a,b),1)=(m_2(1,a),b)=(a,b)
\]
Thus the composition 
\[
W\times W\xrightarrow{m_2} W\xrightarrow{\text{projection}} lh\cong l
\xrightarrow{\Tr} k
\]
coincides with $(-,-)$. 
From Proposition \ref{ref-12.2-75} we deduce that $m_n$ has its image in
$W_c$ for $n\ge 3$.

Put $V=\Sigma\DD \bar{W}$, $V_c=\Sigma \DD W_c$. By Proposition \ref{ref-A.5.4-103} $A$ is weakly
equivalent to $(T_lV,d)$. 
The symmetric bilinear form $(-,-)$ restricted to $W_c$ must be given by contraction 
(in the sense of \eqref{ref-5.4-7}) with some anti-symmetric element
$\eta\in (V_c\otimes_l V_c)_l$ of degree $d-2$. 

Let $z=s^{-1} (h^\ast)\in V$. We must compute $dz$. In other words
we must compute $d_nz$ which is the composition
\begin{equation}
\label{ref-12.2-80}
lz\hookrightarrow V\hookrightarrow TV\xrightarrow{d} TV\xrightarrow{\text{projection}} V^{\otimes n} 
\end{equation}
Dually we must compute
\begin{equation}
\label{ref-12.3-81}
W^{\otimes n}\hookrightarrow BW\xrightarrow{d} BW \xrightarrow{\text{projection}} lh
\end{equation}
We have established that the image of \eqref{ref-12.3-81} is zero when $n\ge 3$.
For $n=2$ is the bilinear form $(-,-)$ which is contraction with $\eta$ (on $\bar{W}$). Dualizing this back to \eqref{ref-12.2-80}
we see that $d_nz=0$ for $n\ge 3$  and $d_2z=\sigma' \eta\sigma''$, 
\end{proof}

\appendix
\section{The bar cobar formalism}
\label{ref-A-82}
\subsection{Weak equivalences}
\label{ref-A.1-83}
We survey the bar cobar formalism for subsequent dualization to the
pseudo-compact case.  We use \cite{hinich,
  Keller12,Lefevre,LodayValette,Positselski} as modern references.
We use some notations that were already introduced in  \S\ref{ref-6-8}.

\medskip

If $C\in \Cog(l)$, $A\in \Alg(l)$ then 
\[
\Hom_{l^e}(\bar{C},\bar{A})
\]
is a DG-vector space and the convolution product $\ast$ makes it into a
DG-algebra. A \emph{twisting cochain} is an element $\tau\in \Hom_{l^e}(\bar{C},\bar{A})_1$
satisfying the Maurer-Cartan equation
\[
d\tau+\tau\ast\tau=0
\]
Let $\Tw(C,A)$ denote the set of twisting cochains in $\Hom_{l^e}(\bar{C},\bar{A})$. 
It is easy to show that $\Tw(-,A)$ 
is representable when restricted to complete augmented
$l$-DG coalgebras.
The representing object is called the \emph{bar construction} on $A$ and is denoted
by $BA$. Likewise  $\Tw(C,-)$ is representable.
The representing object is called the \emph{cobar construction} on $C$
and is denoted by $\Omega C$. Thus we obtain natural isomorphisms 
\begin{equation}
\label{ref-A.1-84}
\Alg(\Omega C,A)\cong\Tw(C,A)\cong \Cog(C,BA)
\end{equation}
(the right one if $C$ is cocomplete).

A weak equivalence between objects in $\Alg(l)$ is defined to be a
quasi-isomorphism. This naive definition does not work for coalgebras.
A
morphism $p:C\r C'$ in $\Cogc(l)$ is said to be
a weak equivalence if $\Omega p:\Omega C\r \Omega C'$ is a
quasi-isomorphism.  
This leads to the following result 
\begin{theorems} \cite[Thm 1.3.12]{Lefevre}.
\label{ref-A.1.1-85}
The functors $(\Omega,B)$ preserve weak equivalences and furthermore they define
inverse equivalences between
the categories 
$\Alg(l)$ and $\Cogc(l)$, localized at weak equivalences. 
\end{theorems}
In particular the counit/unit maps for \eqref{ref-A.1-84}, respectively given by,
\begin{gather}
\label{ref-A.2-86}
\Omega BA\r A\\
\label{ref-A.3-87}
C\r B\Omega C
\end{gather}
are weak equivalences. 

These weak equivalences are part of a model structure on
$\Cogc(l)$ which we will not fully specify. Let us mention however
that every object is cofibrant and the fibrant objects are the $l$-DG-coalgebras
which are cofree when forgetting the differential \cite[\S1.3]{Lefevre}.

A weak equivalence between augmented $l$-DG-coalgebras is a
quasi-isomorphism but not necessarily the other way around (see
\cite[\S1.3.5]{Lefevre} for a counter example). 
This can be repaired in the following
typical case. 
\begin{propositions} 
\label{ref-A.1.2-88}
\cite[Prop.\ 1.3.5.1]{Lefevre}
Assume the gradings on $C,C'\in \Cogc(l)$ are concentrated in degrees $\ge 0$. 
Then a weak equivalence between $C$ and $C'$ is the same as a
quasi-isomorphism. 
\end{propositions} 

For completeness we recall the standard constructions of $BA$ and $\Omega C$.
If $V$ is a graded $l$-bimodule then the tensor algebra $T_lV=\bigoplus_{n\ge 0}
V^{\otimes_l n} $ becomes in a natural way an augmented graded $l$-coalgebra if we put
$\overline{T_l V}= \bigoplus_{n> 0}
V^{\otimes_l n}$ and define the coproduct on $T_lV$ as
\[
\Delta(v_1|\cdots | v_n)=\sum_{i=0,\ldots,n}(v_1| \cdots|
v_i)\otimes (v_{i+1}|\cdots| v_n)
\]
where as customary $(v_1|\cdots|v_n)$ denotes $v_1\otimes\cdots\otimes v_n$
considered as an element of $V^{\otimes_l n}\subset T_lV$ and $()=1$.

If $A$ is an augmented $l$-DG-algebra then $BA=T_l(\Sigma\bar{A})$
with the codifferential $d$ on $T_l(\Sigma \bar{A})$ being defined via
its Taylor coefficients $d_n:(\Sigma \bar{A})^{\otimes n}\hookrightarrow
T(\Sigma\bar{A}) \xrightarrow{d}T(\Sigma\bar{A})\xrightarrow{\text{projection}} \Sigma A$
\begin{equation}
\label{ref-A.4-89}
\begin{aligned}
d_1(sa)&=-sda\\
d_2(sa{\mid} sb)&=(-1)^{|a|}s(ab)\\
d_n&=0\qquad \text{$n\ge 3$}
\end{aligned}
\end{equation}
for $a,b\in A$.  

If $C$ is a DG-$l$-coalgebra
then $\Omega C=T_l(\Sigma^{-1}\bar{C})$ and the differential
is given by 
\begin{equation}
\label{ref-A.5-90}
d(s^{-1}c)=-s^{-1}dc+(-1)^{|c_{(1)}|}(s^{-1}c_{(1)}|s^{-1}c_{(2)})
\end{equation}
for $c\in C$. 
\subsection{Koszul duality}
Let $A\in \Alg(l)$. We recall the standard model structure on $\DGMod(A)$.
\begin{enumerate}
\item The weak equivalences are the quasi-isomorphisms.
\item The fibrations are the surjective maps. 
\item The cofibrations are the maps which have the left lifting
  property with respect to the acyclic fibrations.
\end{enumerate}
It is possible to describe cofibrations more explicitly as retracts of
standard cofibrations but we will not do it. 

Now let $C\in \Cogc(l)$. The following model structure on $\DGComod(C)$ is defined in 
\cite[\S8.2]{Positselski}.
\begin{enumerate}
\item The weak equivalences are the morphisms with a coacyclic cone. 
\item The fibrations are surjective morphisms with kernel which is
  injective when forgetting the differential.
\item The cofibrations are the injective morphisms.
\end{enumerate}
An object is coacyclic if it is in the smallest subcategory of the homotopy category
of $C$ which contains total complexes of short exact sequences and is closed
under arbitrary coproducts. This model structure looks different from
the one defined in \cite[\S2.2.2]{Lefevre}. However both model
structures are Quillen equivalent to the one on $\DGMod(A)$ for $A=\Omega C$, defined above
(see \cite[Thm 2.2.2.2]{Lefevre} and \cite[\S8.4]{Positselski}).  So
they have the same weak equivalences. Since they also have the same cofibrations
they are the same. 

\medskip

We now discuss this Quillen equivalence. 
Let $M\in \DGComod(C^\circ)$ and $N\in \DGMod(A)$. Then $M\otimes_l N$
becomes a left DG-module over $\Hom_{l^e}(\bar{C},\bar{A})$ if we let $\tau\in
\Hom_{l^e}(\bar{C},\bar{A})$ act by
\[
\delta_\tau=(\id \otimes \mu)\circ (\id\otimes \tau\otimes \id)\circ
(\Delta\otimes \id) \in \End(M\otimes_l N)
\]
In particular if $\tau\in \Tw(C,A)$ then $\delta_\tau$ satisfies the
Maurer-Cartan equation in $\End(M\otimes_l N)$.
We let $M\otimes_\tau N$ be equal to $M\otimes_l N$ but with 
$
\delta_\tau
$ added to the differential. 

There exists also an analogue of this construction in case $M\in \DGMod(A^\circ)$
and $N\in \DGComod(C)$. We leave the easy to guess formulas to the reader. 

Here are some useful identities
\begin{align*}
(M\otimes_\tau A)\otimes_A N&=M\otimes_\tau N\\
M\square_C (C\otimes_\tau N)&=M\otimes_\tau N
\end{align*}

There is an analogue of the twisting construction for $\Hom$. Let $M\in \DGComod(C)$
and $N\in \DGMod(A)$. 
Then $\Hom_l(M,N)$ becomes a left DG-module over $\Hom_{l^e}(\bar{C},\bar{A})$ if we let $\tau\in
\Hom_{l^e}(\bar{C},\bar{A})$ act by
\[
\delta_\tau(\phi)=\mu\circ (\tau\otimes\phi)\circ \Delta
\]
If $\tau\in \Tw(C,A)$ then we let $\Hom_\tau(M,N)$ be equal to $\Hom_l(M,N)$
but with $\delta_\tau$ added to the differential. Again this construction may
also be performed with right (co)modules. 

Now we have the following basic identities
\begin{align*}
\Hom_A(A\otimes_\tau M,N)&=\Hom_\tau(M,N)\\
\Hom_C(M,C\otimes_\tau N)&=\Hom_\tau(M,N)
\end{align*}
which yield a pair of adjoint functors \cite[Theorem 2.2.2.2]{Lefevre}
\begin{equation}
\label{ref-A.6-91}
\begin{gathered}
L:\DGComod(C)\r \DGMod(A):M\mapsto A\otimes_\tau M\\
R:\DGMod(A)\r \DGComod(C):N\mapsto C\otimes_\tau N
\end{gathered}
\end{equation}
Below we let $\tau_u$ be the twisting cochain $\bar{C}\r
\overline{\Omega C}$ given by the obvious map. This is the universal
twisting cochain corresponding to the identity map $\Omega C\r \Omega
C$ in \eqref{ref-A.1-84}.  In \cite[\S8.4]{Positselski} it is shown
that in case $A=\Omega C$ and $\tau=\tau_u$ the adjoint pair $(L,R)$
introduced above defines a Quillen equivalence.  In particular a map
$M\r N$ in $\DGComod(C)$ is a weak equivalence if and only if $\Omega
C\otimes_\tau M\r \Omega C\otimes_\tau N$ is a quasi-isomorphism.

The following result is proved in a similar way as Proposition \ref{ref-A.1.2-88}.
\begin{lemmas} \label{ref-A.2.1-92}
Assume that the grading on $C\in \Cogc(l)$ is concentrated in degrees $\ge 0$ and
$M,N\in \DGComod(C)$ are concentrated in degrees $\ge -n$ for certain $n$. Then
a weak equivalence between $M,N$ is the same as a quasi-isomorphism. 
\end{lemmas}

\medskip

\subsection{$A_\infty$-algebras and minimal models}
\label{ref-A.3-93}
By definition a (non-unital)
$l$-$A_\infty$-algebra is an $l$-bimodule~$A$ together with an
$l$-coderivation $d$ of degree one and square zero on the coalgebra
$T_l(\Sigma A)$ compatible with the augmentation. By this we mean
$d(1)=0$, $\epsilon\circ d=0$.
We write $\tilde{B}A=(T_l(\Sigma A),d)$ and call $\tilde{B}A$ the bar
construction of $A$. An $A_\infty$-morphism $A\r A'$ is a DG-coalgebra
morphism $\tilde{B}A\r \tilde{B}A'$.  We write $\Alg_\infty^\bullet(l)$
for the category of $l$-$A_\infty$-algebras.

A coderivation on $T_l(\Sigma A)$ compatible with the augmentation is
determined by ``Taylor coefficients'' ($n\ge 1$)
\[
d_n:(\Sigma A)^{\otimes_l n}\hookrightarrow T_l(\Sigma A)\xrightarrow{d}  T_l(\Sigma A)
\xrightarrow{\text{projection}}
\Sigma \bar{A}
\]
which are of degree one. Introducing suitable signs the $d_n$ may be
transformed into maps
\[
m_n:A^{\otimes_l n}\r A
\]
of degree $2-n$ (see e.g.\ \cite[Lemme
1.2.2.1]{Lefevre}). One has $m_1^2=0$, $m_1$ is a derivation for
$m_2$ and $m_2$ is associative up to a homotopy given by $m_3$. We
view $(A,m_1)$ as a complex and denote its
homology by $H^\ast(A)$. In this way $(H^\ast(A),m_2)$ becomes a graded $l$-algebra
(without unit).

Likewise an $A_\infty$-morphism $f:A\r A'$ is described by maps of degree $1-n$
\[
f_n:A^{\otimes n}\r A'
\]
Here $f_1$ is a morphism of complexes $(A,m_1)\r (A',m'_1)$ 
which is compatible with the multiplications given by $m_2$, $m'_2$ up 
to a homotopy given by $f_2$.  In particular $H^\ast(f_1)$ defines a morphism
of graded $l$-algebras. 

A morphism $f:A\r A'$ in $\Alg_\infty^\bullet(l)$ is said to be a quasi-isomorphism
(or weak equivalence) if $f_1:(A,m_1)\r (A',m'_1)$ is a quasi-isomorphism. 

The
following is a basic result in the theory of $A_\infty$-algebras.
\begin{proposition} \cite[Cor.\ 1.4.14]{Lefevre} \label{ref-A.1-94}
  Let $A\in \Alg_\infty^\bullet(l)$ and let $(H^\ast(A),m_2)$ be its
  cohomology algebra.  Then there exists a structure of an
  $l$-$A_\infty$-algebra on  $H^\ast(A)$ of the form
$(H^\ast(A),m_1=0,m_2,m_3,\ldots)$
  together with a morphism in $\Alg_\infty^\bullet(l)$: $f:H^\ast(A)\r A$
  which lifts the identity $H^\ast(A)\r H^\ast(A)$.
\end{proposition}
An $A_\infty$-algebra with $m_1=0$ is said to be minimal. 
Following Kontsevich one calls the $A_\infty$-algebra $(H^\ast(A),m_1=0,m_2,m_3,\ldots)$
a \emph{minimal model} for $A$. It is unique up to non-unique isomorphism of $l$-$A_\infty$-algebras. 

\medskip

There is an obvious augmented version of the theory of $A_\infty$-algebras. 
An augmented $l$-$A_\infty$-algebra is an $l$-$A_\infty$-algebra $A$ equipped with
a decomposition of $l$-bimodules $A=l\oplus \bar{A}$ such that $\bar{A}$ is
a sub $l$-$A_\infty$-algebra of $A$ and $1\in l$ is a strict unit. I.e.
$m_1(1)=0$, $m_2(1,a)=a$, $m_2(a,1)=a$ and $m_n(\ldots,1,\ldots)=0$ for $n\ge 3$.
Note that the $A_\infty$-structure on $A$ is completely determined by that of $\bar{A}$. 

Likewise an morphism of augmented $A_\infty$-algebras $f:A\r A'$ is a morphism
of $l$-$A_\infty$-algebras that restricts to a morphism of $l$-$A_\infty$-algebras
$\bar{A}\r\bar{A}'$ such that $f_1(1)=1$ and $f_n(\ldots,1,\ldots)=0$ for $n\ge 2$. 
Again $f$ is completely determined by its restriction to $\bar{A}$. 
We denote the category of augmented $l$-$A_\infty$-algebras by $\Alg_\infty(l)$. 

For $A\in \Alg_\infty(l)$ we put $BA=T_l(\Sigma \bar{A})$ and then the
$A_\infty$-structure on $\bar{A}$ defines a codifferential on $BA$ compatible
with the augmentation. Conversely augmented $A_\infty$-algebras may be defined
in terms of codifferentials on $T_l(\Sigma \bar{A})$ 
which are compatible with the augmentation. 

If $A$ is an augmented $l$-$A_\infty$-algebra then there is a (natural) $l$-$A_\infty$-morphism 
$A\r \Omega BA$ to the DG-algebra $\Omega BA$. This morphism is a quasi-isomorphism
(see e.g.\ \cite[Lemma 2.3.4.3]{Lefevre}).  The DG-algebra $\Omega BA$ is called
the \emph{DG-envelope} of $A$. 
\begin{lemmas}
\label{ref-A.3.1-95}
If $A\r A'$ is an $A_\infty$-quasi-isomorphism then $BA\r BA'$
is a weak equivalence. 
\end{lemmas}
\begin{proof} We have to show that $\Omega BA\r \Omega BA'$ is a quasi-isomorphism. 
This follows from the fact that we have we have a commutative diagram
\begin{equation}
\xymatrix{
  A\ar[d]_{\text{qi}}\ar[r]^{\text{qi}} &A'\ar[d]^{qi}\\
\Omega B A\ar[r]&\Omega BA'
}
\end{equation}
\end{proof}

\begin{lemmas}
\label{ref-A.3.2-96}
Assume that $C\in \Cogc(l)$ is weakly equivalent to $(T_l V,d)$. Then there is
an augmented $l$-$A_\infty$-quasi-isomorphism $ l\oplus \Sigma^{-1} V \r \Omega C$.
\end{lemmas} 
\begin{proof} Note that giving the codifferential $d$ on $T_l V$ is
  precisely the same as defining an augmented $l$-$A_\infty$-structure
  on $l+\Sigma^{-1} V$.  As $(T_l V,d)$ is fibrant (see above) the
  weak equivalence $C\r T_lV$ is represented by an actual map of
  augmented $l$-DG-coalgebras. As $(T_l V,d)=B( l+\Sigma^{-1} V)$ we have the following
  quasi-isomorphisms
\[
\Omega C\xrightarrow{DG} \Omega T_l V=\Omega B(l+\Sigma^{-1} V)\xleftarrow{A_\infty} l+\Sigma^{-1} V
\]
The first map is in particular an augmented $l$-$A_\infty$-quasi-isomorphism so it can be inverted 
(e.g.\ \cite[Cor.\ 1.3.1.3]{Lefevre}). This yields what we want. 
\end{proof}

\subsection{The bar cobar formalism in the pseudo-compact case}
\label{ref-A.4-97}
In this paper we use the bar-cobar formalism in the context of
pseudo-compact algebras and modules.  To this end 
we simply dualize everything we have explained above,
using~$\DD$. 
Let $A,C$ be respectively objects in $\PCAlg(l)$ and $\PCCog(l)$.
We put
\begin{equation}
\label{ref-A.8-98}
\begin{aligned}
BA&=\DD \Omega \DD A\\
\Omega C&=\DD B\DD C
\end{aligned}
\end{equation}
We may interpret these definitions more concretely. For $V\in \PCGr(l)$ put
\[
T_l V=\prod_{n\ge 0} V^{\otimes_l n}
\]
One checks that $T_lV$ is naturally a graded augmented pseudo-compact $l$-algebra and
coalgebra. Then $BA=T_l(\Sigma \bar{A})$, $\Omega
C=T_l(\Sigma^{-1}\bar{C})$ with the differentials given by the
formulas \eqref{ref-A.4-89}\eqref{ref-A.5-90}.

We equip $\PCAlgc(l)$ with the dual model structure on $\Cogc(l)$.
In particular 
morphism $p:A\r A'$ in $\PCAlgc(l)$ is a
a weak equivalence if $Bp:BA \r BA'$ is a quasi-isomorphism.
An object is cofibrant if it is of the form $(T_l V,d)$ with
$V\in \PC(l^e)$ and $d$ compatible with the augmentation. 

By similar dualizing we say that a weak equivalence between objects in $\PCCog(C)$ is 
the same as a quasi-isomorphism.

 We equip the categories $\PCDGComod(C)$ and
$\PCDGMod(A)$ with the duals of the model structures on $\DGMod(\DD
C^\circ)$ and $\DGComod(\DD A^\circ)$.

We dualize the functors $L,R$ in the
obvious way: $R=\DD L\DD$, $L=\DD R \DD$. They are given by the same formulas
as \eqref{ref-A.6-91} but now we use them with $C=BA$ and
the universal (continuous) twisting cochain
$\tau_u:\overline{BA}\r \bar{A}$.

A weak equivalence between objects in $\PCDGComod(C)$ is the same as a
quasi-isomorphism. On the other hand a morphism $M\r N$ is
$\PCDGMod(A)$ is a weak equivalence if and only if $BA\otimes_{\tau_u} M\r
BA\otimes_{\tau_u} N$ is a quasi-isomorphism.  The derived categories of
$A$ and $C$ are obtained from $\PCDGMod(A)$ and $\PCDGComod(C)$ by
inverting weak equivalences.
\subsection{Minimal models for pseudo-compact algebras}
\label{ref-A.5-99}
If $d$ is a differential on $T_l W$ with $W\in \PC(l^e)$ then we
will denote its components $W\r W^{\otimes_l n}$ by $d_n$.

We first note that since $\DD T_l W\cong T_l (\DD W)$, specifying a
differential on $T_l W$ is exactly the same as specifying an augmented
$l$-$A_\infty$-structure on $l+\Sigma^{-1} \DD W$ (and this is an
honest $A_\infty$-structure, not a pseudo-compact one).

For $A\in \PCAlgc(l)$ we define the Koszul dual of $A$ as (see also \cite{Keller12})
\begin{equation}
\label{ref-A.9-100}
A^!=\Omega \DD A
\end{equation}
Thus $A^!$ is an honest augmented $l$-DG-algebra (not a pseudo-compact DG-algebra). 
\begin{propositions} \label{ref-A.5.1-101} (Koszul duality, cfr \cite{Keller12}) There
is an equivalence of triangulated categories
\[
D(A)\r D((A^!)^\circ)^\circ
\]
which sends $\Sigma^n l$ to $\Sigma^{-n}A^!$. 
\end{propositions}
\begin{proof} We have
\[
D(A)\cong D(BA)=D(\DD \Omega \DD A)\cong D((\Omega \DD A)^\circ)^\circ
\]
The functor realizing the indicated equivalence is given by
\[
M\mapsto BA\otimes_{\tau_u} M\cong \DD \Omega \DD A\otimes_{\tau_u} M\cong \DD(\DD M\otimes_{\tau_u} A^!)
\mapsto \DD M\otimes_{\tau_u} A^!
\]
We see that $l$ is indeed sent to $A^!$. 
\end{proof}
\begin{corollarys} \label{ref-A.5.2-102}
We have as algebras
\[
\Ext_A^\ast(l,l)\cong H^\ast(A^!)^\circ
\]
\end{corollarys}
\begin{proof}
We have
\begin{align*}
\Ext_A^n(l,l)&=\Hom_{D(A)}(l,\Sigma^n l)\\
&=\Hom_{D(A^{!\circ})^\circ}(A^!,\Sigma^{-n}A^!)\\
&=\Hom_{D(A^{!\circ})}(\Sigma^{-n} A^!,A^!)\\
&=A^!_{n}
\end{align*}
One verifies that this identification inverts the order of the multiplication,
whence the result. 
\end{proof}
\begin{remarks}
One may show that $A^!$ actually computes $\RHom_A(l,l)^\circ$. 
\end{remarks}
\begin{propositions}
\label{ref-A.5.4-103}
Let $A\in \PCAlgc(l)$. Then $A$ there is a weak equivalence $\Omega \DD A^!\r A$. Furthermore
the same holds with $A^!$ replaced by any augmented $l$-$A_\infty$-algebra quasi-isomorphic
to it. Conversely if $A$ is weakly equivalent to $(T_l W,d)$ then  
there is an $A_\infty$-quasi-isomorphism $l+\Sigma^{-1} \DD W\cong A^!$, where
the $A_\infty$-algebra structure on $l+\Sigma^{-1}\DD W$ is as introduced above. 
\end{propositions}
\begin{proof}
We have
\[
\Omega \DD A^!\cong \Omega \DD \Omega \DD A=\Omega B A
\]
and $\Omega B A$ is weakly equivalent to $A$ by applying $\DD$ to
\eqref{ref-A.3-87}.  This implies that $A$ is weakly equivalent to $\Omega \DD A^!$. 
The fact that $A^!$ may be replaced by any other algebra
quasi-isomorphic to it follows from the fact that $\Omega \DD A^!=\DD
BA^!$ combined with Lemma \ref{ref-A.3.1-95}.

Finally by applying $\DD$ to the conclusion of Lemma \ref{ref-A.3.2-96} with
$C=\DD A$ and $V=\DD W$ we obtain $A^!\cong l+\Sigma^{-1} \DD W$.
\end{proof}
\begin{corollarys}
\label{ref-A.5.5-104}
Let $A\in \PCAlgc(l)$. 
Then there exists a weak equivalence
$(T_l W,d)\r A$ such that $d_1=0$. 
\end{corollarys}
\begin{proof}
We let  $l+\Sigma^{-1} W$ be a minimal augmented
$l$-$A_\infty$-model for $A^!$.  Then from Proposition \ref{ref-A.5.4-103}
 obtain that $A$ is weakly equivalent
to $(T_l W,d)$ where $d_1=0$.
\end{proof}
Following traditional terminology we call a weak equivalence as in 
Corollary \ref{ref-A.5.5-104} a \emph{minimal model} 
for $A$. 
\begin{corollarys}
\label{ref-A.5.6-105}
Whenever we have a minimal model $T_l W\r A$ then $W\cong \Sigma^{-1}(\DD\Ext^\ast_A(l,l))_{\le 0}$
and the $m_2$ multiplication on $l+\Sigma^{-1} \DD W\cong \Ext^\ast_A(l,l)$  for the induced
$A_\infty$-structure corresponds
to the opposite of the Yoneda multiplication on $\Ext^\ast_A(l,l)$. 
\end{corollarys}
\begin{proof}  By Proposition \ref{ref-A.5.4-103} we have an $A_\infty$-quasi-isomorphism
$l+\Sigma^{-1} \DD W\cong A^!$ and hence an isomorphism as algebras
\[
l+\Sigma^{-1} \DD W\cong H^\ast(l+\Sigma^{-1}\DD W)\cong H^\ast(A^!)
\]
It now suffices to apply Corollary \ref{ref-A.5.2-102}. 
\end{proof}
\section{Hochschild homology of pseudo-compact algebras}
\label{ref-B-106}
Let $A\in \PCAlgc(l)$. It is easy to see that the tensor product $-\otimes_A-$  satisfies 
the hypotheses of  \cite[Prop.\ 4.1]{quillen1} in both arguments and hence it 
may be left derived in both arguments. It is also easy to see that  deriving
the first argument gives the same result as deriving the second argument. Therefore
we make no distinction between the two and write the result as $-\Lotimes_A-$. 

Now we work over $A^e$ which is considered as an object in $\Mod(l^e)$. 
Our aim is to show the following result
\begin{proposition}
\label{ref-B.1-107} If $A\in \PCAlgc(l)$ then
  $A\Lotimes_{A^e} A$ is computed by the standard Hochschild
  complex $(\C(A),b)= ((A\otimes_l T_l(\Sigma A))_l,d_A+d_{\text{Hoch}})$
  where $d_{\text{Hoch}}$ is
  the usual Hochschild differential.
\end{proposition}
\begin{proof}
In lemma \ref{ref-B.2-109} below we show that $A\otimes_{\tau_u} BA\otimes_{\tau_u} A$ is a cofibrant replacement for $A$  in $\PCDGMod(A^e)$.

Thus we have the following formula 
\[
A\Lotimes_{A^e} A=A\otimes_{A^e} (A\otimes_{\tau_u} BA \otimes_{\tau_u} A)
\]
We have
\begin{equation}
\label{ref-B.1-108} 
A\otimes_{A^e} (A\otimes_{\tau_u} BA\otimes_{\tau_u} A)\cong ((A\otimes_l T_l(\Sigma\bar{A}))^l,d_A+d_{\text{Hoch}})
\end{equation}
where $d_{\text{Hoch}}$ is the usual Hochschild differential. 

The righthand side of \eqref{ref-B.1-108} is the reduced Hochschild complex. It is quasi-isomorphic to  the standard Hochschild
complex which has the form
\[
((A\otimes_l T_l(\Sigma A))^l,d_A+d_{\text{Hoch}})
\]
(see \cite[Prop. 1.6.5]{Loday1}).
\end{proof}
\begin{lemma}
\label{ref-B.2-109}
$A\otimes_{\tau_u} BA\otimes_{\tau_u} A$ is a cofibrant replacement for $A$  in $\PCDGMod(A^e)$
\end{lemma}
\begin{proof}
We have a Quillen equivalence
\begin{equation}
\label{ref-B.2-110}
\begin{gathered}
L^e:\PCDGComod((BA)^e)\r \PCDGMod(A^e):N\mapsto A\otimes_{\tau_u} N\otimes_{\tau_u} A\\
R^e:\PCDGMod(A^e)\r \PCDGComod(BA^e):M\mapsto BA\otimes_{\tau_u} M\otimes_{\tau_u} BA
\end{gathered}
\end{equation}
As $A\otimes_{\tau_u} BA\otimes_{\tau_u} A$ is a projective bimodule when forgetting
the differential it is cofibrant.

Hence we have to show that
\[
\mu_{13}:A\otimes_{\tau_u} BA\otimes_{\tau_u} A\r A:a\otimes b\otimes c\mapsto a\epsilon(b)c
\]
is a weak equivalence in $\PCDGMod(A^e)$. By the  Quillen equivalence 
between $\PCDGMod(A^e)$ and $\PCDGComod(BA^e)$ 
we may as well show that the adjoint map
\begin{equation}
\label{ref-B.3-111}
\Delta_{13}:BA \r BA\otimes_{\tau_u} A \otimes_{\tau_u} BA:c\mapsto c_{(1)}\otimes 1 \otimes c_{(2)}
\end{equation}
is a weak equivalence, or equivalently a quasi-isomorphism of $BA$-bi-comodules. If we view \eqref{ref-B.3-111}
as a map of \emph{left} comodules then it is precisely the unit map
\[
BA\r RL(BA)
\]
which is a weak equivalence (and hence quasi-isomorphism) since $(L,R)$ forms a Quillen equivalence. 
\end{proof}
\section{Symmetry for Hochschild homology}
\label{ref-C-112}
Assume that $A$ is an $l$-algebra and let $M$ be a finitely generated projective $A$-bimodule. Put $M^D=\Hom_{A^e}(M,A\otimes A)$.
Then an element $\xi\in M\otimes_{A^e} M$ defines a bimodule map
\[
\xi^+:M^D\r M:\phi\mapsto \phi(\xi')''\xi'' \phi(\xi')'
\]
and conversely using the identification
\[
\Hom_{A^e}(M^D,M)\cong M\otimes_{A^e} M
\]
any bimodule morphism $M^D\r M$ is of the form $\xi^+$ for some $\xi\in M\otimes_{A^e}M$.

\medskip

There is a $\ZZ/2\ZZ=\{1,\ss\}$-action on $M\otimes_{A^e} M$ such that $\ss(a\otimes b)=b\otimes a$.
One checks that
\[
\ss(\xi)^+=c\circ (\xi^+)^D
\]
for the canonical isomorphism $c:M\mapsto M^{DD}:m\mapsto (\phi\mapsto \phi(m)''\otimes \phi(m)')$

Hence if $\xi$ is symmetric and we view $c$ as an identification then $(\xi^+)^D=\xi^+$.

\medskip

What we have just explained extends to the case where $A$ is an $l$-DG-algebra and $M$ is 
a perfect object in $D(A^e)$
(where we now use the derived version of $(-)^D$ as introduced in \S\ref{ref-8-17}). We will
apply it in the case $M=A$. We will prove the following result
\begin{proposition}
$H_d(\ss)$ acts trivially on $\HH_d(A)=H_d(A\Lotimes_{A^e}A)$. Hence if $A$ is homologically smooth
then any $\eta:A^D\r {\Sigma^{-d}} A$ is automatically self dual. 
\end{proposition}
\begin{proof}
 For our purpose we may and we will 
assume that $A$ is cofibrant.  We will use the complex $Y(A)
=\Sigma \Omega_l^1 A\otimes \Sigma \Omega_l^1 A\oplus \Sigma (\Omega_l^1 A)\oplus
\Sigma(\Omega_l^1 A)_l \oplus (A\otimes A)_l$
 to compute
$A\Lotimes_{A^e} A$ (see \eqref{ref-11.4-54}). 

Taking homology for rows and columns in $Y(A)$ we get two maps 
\[
l,r:Y(A)\r X(A)
\]
where by a slight abuse of notation we have written $X(A)$ for
$\cone((\Omega^1_lA)_\natural\xrightarrow{\partial_1} A_l)$ (see \S\ref{ref-7.2-12}). Note that by Proposition
\ref{ref-7.2.1-15} $X(A)$ computes the Hochschild homology of $A$. 

We claim that $l,r$ are homotopy equivalent. 
To prove this we will describe $l$ and $r$ explicitly:
\begin{align*}
l(s\omega_1\otimes s\omega_2)&=0&& (\text{on $\Sigma\Omega_l^1 A\otimes_{A^e}\Sigma\Omega_l^1 A$})\\
l(s\omega_1)&=0&& \text{(on the first copy of $(\Sigma\Omega_l^1 A)_l$)}\\
l(s\omega_2)&=s\omega_{2,\natural}&&\text{(on the second copy of $\Sigma(\Omega_l^1 A)_l$)}
\\
l(a\otimes b)&=\overline{ab}&& (\text{on $(A\otimes_l A)_l$})
\end{align*}
taking into account
that in \eqref{ref-11.4-54-1}, $s\omega_1$ is represented by $s\omega_1\otimes (1\otimes 1)$,
$s\omega_2$ is represented by $(1\otimes 1)\otimes s\omega_2$, $a\otimes b$ is
represented by $(a\otimes b)\otimes (1\otimes 1)$ and taking homology
for rows/columns corresponds to taking homology in the first/second factor.
For the last line one needs to take into account the identification \eqref{ref-11.5-55}.

Likewise we have
\begin{align*}
r(s\omega_1\otimes s\omega_2)&=0&& (\text{on $\Sigma\Omega_l^1 A\otimes_{A^e}\Sigma\Omega_l^1 A$})\\
r(s\omega_1)&=s\omega_{1,\natural}&& \text{(on the first copy of $(\Sigma\Omega_l^1 A)_l$)}\\
r(s\omega_2)&=0&&\text{(on the second copy of $\Sigma(\Omega_l^1 A)_l$)}\\
r(a\otimes b)&=(-1)^{|a||b|}\overline{ba}&& (\text{on $(A\otimes_l A)_l$})
\end{align*}

Thus for the difference $m=l-r$
\begin{align*}
m(s\omega_1\otimes s\omega_2)&=0\\
m(s\omega_1)&=-s\omega_{1,\natural}\\
m(s\omega_2)&=s\omega_{2,\natural}\\
m(a\otimes b)&=\overline{[a,b]}
\end{align*}
Now we define a map of degree $-1$ 
\[
h:Y(A)\r X(A)
\]
by
\begin{align*}
h(s\omega_1\otimes s\omega_2)&=0\\
h(s\omega_1)&=0\\
h(s\omega_2)&=0\\
h(a\otimes b)&=s(aDb)_{\natural}
\end{align*}
Now we compute $dh=[d,h]=\partial_1\circ h+h\circ(\partial_1\otimes 1+1\otimes\partial_1)$.
To this end we have to know $\partial_1^{\text{hor}}$ and $\partial_1^{\text{ver}}$.
We compute
\begin{align*}
\partial_1^{\text{hor}}(aDb)&=(\partial_1\otimes 1)(aDb\otimes (1\otimes 1))\\
&=(ab\otimes 1-a\otimes b) \otimes (1\otimes 1)\\
&=ab\otimes 1-a\otimes b
\end{align*}
\begin{align*}
\partial_1^{\text{ver}}(aDb)&=(1\otimes \partial_1)((1\otimes 1)\otimes aDb)\\
&=(1\otimes 1)\otimes (ab\otimes 1-a\otimes b)\\
&=1\otimes ab-(-1)^{|a||b|}b\otimes a
\end{align*}
taking into account the identification \eqref{ref-11.5-55}.

We now find 
\[
(dh)(s\omega_1\otimes s\omega_2)=0
\]
\begin{align*}
(dh)(s(aDb))&=h(ab\otimes 1-a\otimes b)\\
&=-s(aDb)_{\natural}&\text{(on the first copy of $\Sigma \Omega^1_l A$)}
\end{align*}
\begin{align*}
(dh)(s(aDb))
&=h((1\otimes ab-(-1)^{|a||b|}b\otimes a)\\
&=s(Dab)_{\natural}-(-1)^{|a||b|}s(bDa)_{\natural}\\
&=s(aDb)_{\natural}&\text{(on the second copy of $\Sigma\Omega^1_l A$)}
\end{align*}
\begin{align*}
(dh)(a\otimes b)&=\partial_1(s(aDb)_{\natural})&\\
&=\overline{[a,b]}
\end{align*}
Hence $h$ is indeed a homotopy connecting $l$ and $r$. 

\medskip

Now we have the following commutative diagram of complexes.
\[
\xymatrix{
Y(A) \ar[r]^\ss\ar[d]_l & Y(A)\ar[d]^r\\
X(A) \ar@{=}[r]& X(A)
}
\]
The top line comes from the fact that $Y(A)$ is obtained from tensoring the
bimodule resolution of $A$ with
itself over $A^e$.

Taking homology we obtain 
\[
\xymatrix{
H_d(Y(A)) \ar[r]^{H_d(\ss)}\ar[d]_{H_d(l)} & H_d(Y(A))
\ar[d]^{H_d(r)}\\
\HH_d(A) \ar@{=}[r]& \HH_d(A)
}
\]
Since $r$ and $l$ are homotopy we have $H_d(l)=H_d(r)$ and we are done. 
\end{proof}
\section{Koszul duality for Hochschild homology}
\label{ref-D-113}
The definition of the Hochschild mixed complex may be dualized to coalgebras.
If $C$ is a counital $l$-DG-coalgebra then the Hochschild mixed complex of
$C$ is $(\C(C),b,B)$ where $(\C(C),b)$ is the sum total complex of a
double complex of the form
\[
0\r C^l \xrightarrow{\partial} (C\otimes_l C)^l \xrightarrow{\partial} (C\otimes_l C\otimes_l C)^l\r \cdots
\]
where $l$ denotes the centralizer and $\partial$ is the dual of the
Hochschild differential.  $B$ is the dual of the Connes differential.
There
exist a similar normalized mixed complex denoted by $(\bar{\C}(C),b,B)$. The mixed
complexes $\bar{\C}(C)$ and
$\C(C)$ are quasi-isomorphic (see \cite[Prop. 1.6.5]{Loday1}). 

The following result is well known. 
\begin{proposition}
\label{ref-D.1-114}
Let $C\in \Cogc(l)$.  There is a quasi-isomorphism of mixed complexes.
\begin{equation}
\label{ref-D.1-115}
\C(\Omega C)\r \C(C)
\end{equation}
\end{proposition}
\begin{proof}
Put $A=\Omega C$.
Since $C$ is cocomplete $A$ is cofibrant. Hence we may apply
Proposition \ref{ref-7.2.1-15} to obtain a quasi-isomorphism
\[
\C(A)\r M X(A)
\]
By the dual version of \cite[Theorem 4]{Quillen2} we have an isomorphism of
complexes
\[
(\Omega_l^1 A)_\natural\cong \Sigma^{-1}\C(\bar{C})
\]
By definition $MX(A)=\cone( (\Omega_l^1 A)_\natural\xrightarrow{\partial_1} A_l)$.
Since $C$ is augmented we have $C=l\oplus \bar{C}$. We get isomorphisms
as graded vector spaces
\begin{align*}
\cone( (\Omega_l^1 A)_\natural\xrightarrow{\partial_1} A_l)&=\C(\bar{C})\oplus A_l \\
&=\C(\bar{C})\oplus T_l(\Sigma^{-1}\bar{C})\\
&=\bar{\C}(C)
\end{align*}
One checks that this isomorphism is compatible with $(b,B)$ and hence
 yields an isomorphism of mixed complexes
\[
M X(A)\cong (\bar{\C}(C),b,B)
\]
Combining this with the standard quasi-isomorphism of mixed complexes 
(see \cite[Prop. 1.6.5]{Loday1})
\[
(\bar{\C}(C),b,B)\r (\C(C),b,B)
\]
yields indeed a quasi-isomorphism as in \eqref{ref-D.1-115}.
\end{proof}
\begin{corollary} \label{ref-D.2-116}
Let $A\in \PCAlgc(l)$. Then we have a quasi-isomorphism of mixed complexes
\begin{equation}
\label{ref-D.2-117}
C(A^!)\r \DD \C(A)
\end{equation}
This quasi-isomorphism is
natural in $A$ (taking into account that $A\mapsto A^!$ is a contravariant functor). 
\end{corollary}
\begin{proof} This follows the fact that $\DD \C(A)=\C(\DD A)$ together with
Proposition \ref{ref-D.1-114}. The naturality of \eqref{ref-D.2-117} follows
from the naturality of \eqref{ref-D.1-115}.
\end{proof}
\begin{corollary}
\label{ref-D.3-118}
  A   a weak equivalence $A\r A'$ in $\PCAlgc(l)$ induces a
  quasi-isomorphism $\C(A)\r \C(A')$ of mixed complexes. 
\end{corollary}
\begin{proof} By definition of $(-)^!$ we obtain that $A^{\prime !}\r A^{!}$ is a
quasi-isomorphism in $\Alg(l)$. Hence this induces a quasi-isomorphism
 $\C(A^{\prime !})\r \C(A^{!})$. By Corollary \ref{ref-D.2-116} we
get a commutative diagram of mixed complexes
\[
\xymatrix{\C(A^{\prime !})\ar[d]_\cong \ar[r]^\cong & \DD C(A')\ar[d]\\
\C(A^{!})\ar[r]_\cong & \DD\C(A)
}
\]
So the rightmost map is indeed a quasi-isomorphism.
\end{proof}
\begin{corollary} \label{ref-D.4-119} Assume that $A\r A'$ is a weak
  equivalence in $\PCAlgc(l)$ between homologically smooth algebras. Then $A$ is exact $d$ Calabi-Yau if and
  only if this is the case for $A'$.
\end{corollary}
\begin{proof} Taking into account Corollary \ref{ref-D.3-118} we have to prove 
that $\eta\in \HH_d(A)$ is non-degenerate if and only if its image in $\HH_d(A')$
is non-degenerate. This is a formal verification which we leave to the reader. 
\end{proof}

\begin{thebibliography}{10}

\bibitem{Amiot}
C.~Amiot, \emph{Cluster categories for algebras of global dimension 2 and
  quivers with potential}, Ann. Inst. Fourier (Grenoble) \textbf{59} (2009),
  no.~6, 2525--2590.

\bibitem{BSW}
R.~Bocklandt, T.~Schedler, and M.~Wemyss, \emph{Superpotentials and higher
  order derivations}, J. Pure Appl. Algebra \textbf{214} (2010), no.~9,
  1501--1522.

\bibitem{Bocklandt}
Raf Bocklandt, \emph{Graded {C}alabi {Y}au algebras of dimension 3}, J. Pure
  Appl. Algebra \textbf{212} (2008), no.~1, 14--32.

\bibitem{LebBock}
Raf Bocklandt and Lieven Le~Bruyn, \emph{Necklace {L}ie algebras and
  noncommutative symplectic geometry}, Math. Z. \textbf{240} (2002), no.~1,
  141--167. \MR{MR1906711 (2003g:16013)}

\bibitem{CBEG}
William Crawley-Boevey, Pavel Etingof, and Victor Ginzburg,
  \emph{Noncommutative geometry and quiver algebras}, Adv. Math. \textbf{209}
  (2007), no.~1, 274--336.

\bibitem{CQ}
Joachim Cuntz and Daniel Quillen, \emph{Algebra extensions and nonsingularity},
  J. Amer. Math. Soc. \textbf{8} (1995), no.~2, 251--289.

\bibitem{Davison}
B.~Davison, \emph{Superpotential algebras and manifolds}, Adv. Math.
  \textbf{231} (2012), no.~2, 879--912.

\bibitem{VdBdT1}
L.~de~Thanhoffer~de V{\"o}lcsey and M.~Van~den Bergh, \emph{Explicit models for
  some stable categories of maximal {C}ohen-{M}acaulay modules},
  arXiv:1006.2021v1.

\bibitem{VdBdT2}
\bysame, \emph{Some new examples of nondegenerate quiver potentials}, Int.
  Math. Res. Not. IMRN (2013), no.~20, 4672--4686.

\bibitem{DWZ}
H.~Derksen, J.~Weyman, and A.~Zelevinsky, \emph{Quivers with potentials and
  their representations. {I}. {M}utations}, Selecta Math. (N.S.) \textbf{14}
  (2008), no.~1, 59--119.

\bibitem{DWZ2}
\bysame, \emph{Quivers with potentials and their representations {II}:
  applications to cluster algebras}, J. Amer. Math. Soc. \textbf{23} (2010),
  no.~3, 749--790.

\bibitem{Gabriel}
Pierre Gabriel, \emph{Des cat\'egories ab\'eliennes}, Bull. Soc. Math. France
  \textbf{90} (1962), 323--448.

\bibitem{Gi}
G.~Ginzburg, \emph{Calabi-{Y}au algebras}, arXiv:math/0612139.

\bibitem{Ginzburg1}
Victor Ginzburg, \emph{Non-commutative symplectic geometry, quiver varieties,
  and operads}, Math. Res. Lett. \textbf{8} (2001), no.~3, 377--400.
  \MR{MR1839485 (2002m:17020)}

\bibitem{Goodwillie}
T.~G. Goodwillie, \emph{Cyclic homology, derivations, and the free loopspace},
  Topology \textbf{24} (1985), no.~2, 187--215.

\bibitem{hinich}
V.~Hinich, \emph{Descent of {D}eligne groupoids}, Internat. Math. Res. Notices
  (1997), no.~5, 223--239.

\bibitem{Keller}
B.~Keller, \emph{Introduction to ${A}$-infinity algebras and modules}, Homology
  Homotopy Appl. \textbf{3} (2001), no.~1, 1--35. \MR{1 854 636}

\bibitem{Keller12}
\bysame, \emph{Koszul duality and coderived categories},
  http://www.math.jussieu.fr/\~{}keller/publ/kdc.dvi, 2003.

\bibitem{Keller11}
\bysame, \emph{Deformed {C}alabi-{Y}au completions}, J. Reine Angew. Math.
  \textbf{654} (2011), 125--180, With an appendix by Michel Van den Bergh.

\bibitem{KY}
B.~Keller and D.~Yang, \emph{Derived equivalences from mutations of quivers
  with potential}, Adv. Math. \textbf{226} (2011), no.~3, 2118--2168.

\bibitem{KS2}
M.~Kontsevich and Y.~Soibelman, \emph{Notes on {$A_\infty$}-algebras,
  {$A_\infty$}-categories and non-commutative geometry}, Homological mirror
  symmetry, Lecture Notes in Phys., vol. 757, Springer, Berlin, 2009,
  pp.~153--219.

\bibitem{Kosymp}
Maxim Kontsevich, \emph{Formal (non)commutative symplectic geometry}, The
  Gel\cprime fand Mathematical Seminars, 1990--1992, Birkh\"auser Boston,
  Boston, MA, 1993, pp.~173--187. \MR{MR1247289 (94i:58212)}

\bibitem{Lazaroiu}
C.~I. Lazaroiu, \emph{On the non-commutative geometry of topological
  {D}-branes}, J. High Energy Phys. (2005), no.~11, 032, 57 pp. (electronic).

\bibitem{Lefevre}
K.~Lef\`evre-Hasegawa, \emph{Sur les {$A_\infty$}-cat\'egories}, Ph.D. thesis,
  Universit\'e Paris 7, 2003.

\bibitem{LodayValette}
J.-L. Loday and B.~Vallette, \emph{Algebraic operads}, Grundlehren der
  Mathematischen Wissenschaften [Fundamental Principles of Mathematical
  Sciences], vol. 346, Springer, Heidelberg, 2012.

\bibitem{Loday1}
Jean-Louis Loday, \emph{Cyclic homology}, second ed., Grundlehren der
  Mathematischen Wissenschaften [Fundamental Principles of Mathematical
  Sciences], vol. 301, Springer-Verlag, Berlin, 1998, Appendix E by Mar\'\i a
  O. Ronco, Chapter 13 by the author in collaboration with Teimuraz Pirashvili.

\bibitem{Mont2}
S.~Montgomery, \emph{Hopf algebras and their actions on rings}, CBMS Regional
  Conference Series in Mathematics, vol.~82, Published for the Conference Board
  of the Mathematical Sciences, Washington, DC; by the American Mathematical
  Society, Providence, RI, 1993.

\bibitem{Positselski}
L.~Positselski, \emph{Two kinds of derived categories, {K}oszul duality, and
  comodule-contramodule correspondence}, Mem. Amer. Math. Soc. \textbf{212}
  (2011), no.~996, vi+133.

\bibitem{quillen1}
D.~Quillen, \emph{Rational homotopy theory}, Ann. of Math. (2) \textbf{90}
  (1969), 205--295.

\bibitem{Quillen2}
\bysame, \emph{Algebra cochains and cyclic cohomology}, Inst. Hautes \'Etudes
  Sci. Publ. Math. (1988), no.~68, 139--174 (1989).

\bibitem{segal}
E.~Segal, \emph{The {$A_\infty$} deformation theory of a point and the derived
  categories of local {C}alabi-{Y}aus}, J. Algebra \textbf{320} (2008), no.~8,
  3232--3268.

\bibitem{VdB33}
M.~Van~den Bergh, \emph{Double {P}oisson algebras}, to appear in Trans. Amer.
  Math. Soc.

\bibitem{VdB36}
\bysame, \emph{Non-commutative quasi-{H}amiltonian spaces}, arXiv:math/0703293.

\bibitem{VdB19}
\bysame, \emph{Blowing up of non-commutative smooth surfaces}, Mem. Amer. Math.
  Soc. \textbf{154} (2001), no.~734, x+140. \MR{1 846 352}

\end{thebibliography}
\def\cprime{$'$} \def\cprime{$'$} \def\cprime{$'$}
\providecommand{\bysame}{\leavevmode\hbox to3em{\hrulefill}\thinspace}
\providecommand{\MR}{\relax\ifhmode\unskip\space\fi MR }
\providecommand{\MRhref}[2]{%
  \href{http://www.ams.org/mathscinet-getitem?mr=#1}{#2}
}
\providecommand{\href}[2]{#2}

\end{document}